\documentclass[10pt]{amsart}
\usepackage{amsmath,amscd,amssymb,amsfonts}
\usepackage{mathrsfs}
\usepackage{hyperref}
\numberwithin{equation}{section}       
\numberwithin{figure}{section}       

\theoremstyle{plain}

\newtheorem{prop}{Proposition}[section]
\newtheorem{coro}[prop]{Corollary}
\newtheorem{lemm}[prop]{Lemma}

\newtheorem{theoalph}{Theorem}

\newtheorem{conj}[prop]{Conjecture}
\newtheorem{prob}[prop]{Problem}

\theoremstyle{definition}
\newtheorem{defi}[prop]{Definition}

\theoremstyle{remark}
\newtheorem{rema}[prop]{Remark}

\newtheoremstyle{citing}
  {3pt}
  {3pt}
  {\itshape}
  {}
  {\bfseries}
  {.}
  {.5em}
  {\thmnote{#3}}

\theoremstyle{citing}
\newtheorem*{generic}{}

\newcommand{\partn}[1]{{\smallskip \noindent \textbf{#1.}}}

\DeclareMathAlphabet{\mathpzc}{OT1}{pzc}{m}{it} 

\newcommand{\C}{\mathbb{C}}

\newcommand{\F}{\mathbb{F}}

\newcommand{\Q}{\mathbb{Q}}

\newcommand{\cO}{\mathcal{O}}

\newcommand{\sK}{\mathscr{K}}

\newcommand{\hDelta}{\widehat{\Delta}}

\newcommand{\tK}{\widetilde{K}}

\newcommand{\tP}{\widetilde{P}}

\newcommand{\teta}{\widetilde{\teta}}

\newcommand{\tlambda}{\widetilde{\lambda}}

\renewcommand{\=}{ : = }

\newcommand{\OK}{{\cO_K}}
\newcommand{\mK}{{\mathfrak{m}_K}}

\newcommand{\wtf}{\widetilde{f}}
\newcommand{\whg}{\widehat{g}}
\newcommand{\whA}{\widehat{A}}
\newcommand{\whB}{\widehat{B}}

\DeclareMathOperator{\wideg}{wideg}
\DeclareMathOperator{\ord}{ord}

\DeclareMathOperator{\resit}{r\acute{e}sit}

\begin{document}

\title[Optimal cycles in ultrametric dynamics]{Optimal cycles in ultrametric dynamics and minimally ramified power series}

\author[Karl-Olof Lindahl]{Karl-Olof Lindahl$^\dag$}
\thanks{$^\dag$ This research was supported by the Swedish Research and Education in Mathematics Foundation SVeFUM
and by MECESUP2 (PUC-0711) and DICYT (U. de Santiago de Chile).}
\address{Department of Mathematics, Linn\ae us University, 351 95, V\"{a}xj\"{o}, Sweden}
\email{karl-olof.lindahl@lnu.se}
\urladdr{\url{http://karl-lindahl.com/}}

\author[Juan Rivera-Letelier]{Juan Rivera-Letelier$^\ddag$}
\thanks{$^\ddag$ This research was partially supported by FONDECYT grant 1100922, Chile}
\address{Juan Rivera-Letelier, Facultad de Matem{\'a}ticas, Pontificia Universidad Cat{\'o}lica de Chile, Avenida Vicu{\~n}a Mackenna~4860, Santiago, Chile}
\email{riveraletelier@mat.puc.cl}
\urladdr{\url{http://rivera-letelier.org/}}

\begin{abstract}
We study ultrametric germs in one variable having an irrationally indifferent fixed point at the origin with a prescribed multiplier.
We show that for many values of the multiplier, the cycles in the unit disk of the corresponding monic quadratic polynomial are ``optimal'' in the following sense: They minimize the distance to the origin among cycles of the same minimal period of normalized germs having an irrationally indifferent fixed point at the origin with the same multiplier.
We also give examples of multipliers for which the corresponding quadratic polynomial does not have optimal cycles.
In those cases we exhibit a higher degree polynomial such that all of its cycles are optimal.
The proof of these results reveals a connection between the geometric location of periodic points of ultrametric power series and the lower ramification numbers of wildly ramified field automorphisms.
We also give an extension of
Sen's theorem on wildly ramified field automorphisms, and a
characterization of minimally ramified power series in terms of the iterative residue.
\end{abstract}


\maketitle

%
%

\section{Introduction}
\label{s:introduction}

In proving the optimality of the Bruno condition for the local linearization of fixed points of holomorphic germs, Yoccoz showed the following dichotomy for quadratic polynomials of the form
\begin{equation}
\label{e:quadratic}
P_\lambda(z) \= \lambda z + z^2,
\end{equation}
where the complex number~$\lambda$ satisfies~$|\lambda| = 1$ and is not a root of unity: Either~$P_\lambda$ is locally linearizable at~$z = 0$, or every neighborhood of~$z = 0$ contains a periodic cycle different than~$z = 0$, see~\cite{Yoc95b}.
This last property is usually known as the \emph{small cycles property}, and it is clearly an obstruction for local linearization.
In fact, in the case~$P_{\lambda}$ is not locally linearizable at~$z = 0$, Yoccoz proved more: The distance of a small cycle of~$P_{\lambda}$ to~$z = 0$ is essentially the smallest possible among cycles of the same minimal period of normalized holomorphic germs of the form
\begin{equation}
  \label{e:series}
  f(z) = \lambda z + \cdots,
\end{equation}
see~\cite[\S$6.6$]{Yoc95b}.

In this paper we prove an analogous result over an arbitrary ultrametric field.
When the residue characteristic of the ground field is odd and~$0 < |\lambda - 1| < 1$, we show that every cycle of the quadratic polynomial~\eqref{e:quadratic} that is in the open unit disk and that has minimal period at least~$2$, is ``optimal'' in the following sense: It minimizes the distance to~$z = 0$ among cycles of the same minimal period of normalized germs of the form~\eqref{e:series}, see Theorem~\ref{t:quadratic optimality and minimality} in~\S\ref{ss:optimal cycles} and the remarks that follow.
When either the residue characteristic of the field is~$2$, or~$|\lambda - 1| = 1$, the quadratic polynomial~\eqref{e:quadratic} does not necessarily have this property.
In this case we find a higher degree polynomial such that all of its cycles in the open unit disk are optimal, see Theorem~\ref{t:periodic optimality} in~\S\ref{ss:optimal cycles}.
A consequence of these results is that irrationally indifferent periodic points are isolated (Corollary~\ref{c:irrational are isolated} in~\S\ref{ss:periodic bound}).
This is new in positive characteristic (in characteristic zero it follows from the local linearization result of Herman and Yoccoz~\cite{HerYoc83}).

The proof of these results reveals a connection between the geometric location of periodic points of ultrametric power series and the lower ramification numbers of wildly ramified field automorphisms, as studied by Sen~\cite{Sen69}, Keating~\cite{Kea92}, Laubie and Sa{\"{\i}}ne~\cite{LauSai98}, Wintenberger~\cite{Win04}, and others.
In fact, we show that for a generic power series of the form~$f(z) = \lambda z + \cdots$, normalized so that it has integer coefficients, the existence of an optimal cycle is equivalent to the reduction of~$f$ having the least possible lower ramification numbers, see Theorem~\ref{t:optimality and minimality} in~\S\ref{ss:minimally ramified intro}.
Such ``minimally ramified'' power series were previously considered by Laubie, Movahhedi, and Salinier in~\cite{LauMovSal02}, in their study of Lubin's conjecture~\cite{Lub94}.

In proving our main results, we give an extension of the main theorem
of Sen in~\cite{Sen69} (Theorem~\ref{t:higher order sen}
in~\S\ref{ss:higher order sen}), and give a characterization of
minimally ramified power series in terms of the iterative residue,
which is a conjugacy invariant introduced by {\'E}calle in the complex setting (Theorem~\ref{t:characterization of minimally ramified} in~\S\ref{s:characterization of minimally ramified}).

We now proceed to describe our main results in more detail.

\subsection{Periodic points of normalized power series}
\label{ss:periodic of normalized}
Let~$(K, | \cdot|)$ be an ultrametric field and denote by~$\OK$ the ring of integers of~$K$, by~$\mK$ the maximal ideal of~$\OK$, and by~$\tK \= \OK / \mK$ the residue field of~$K$.

Let~$\lambda$ in~$K$ be such that~$|\lambda| = 1$, and let
\begin{equation}
  \label{e:germ}
f(z) = \lambda z + \cdots  
\end{equation}
be a power series in~$K[[z]]$ converging on a neighborhood of~$z = 0$.
Through a scale change, we can assume~$f$ has coefficients in~$\OK$.
We say a power series~$f$ as in~\eqref{e:germ} is \emph{normalized} if
it has coefficients in~$\OK$.
When the ground field~$K$ is algebraically closed, the power~$f$ is
normalized if and only if it converges and is univalent on the open unit disk~$\mK$, see for example~\cite[\S$1.3$]{Rivthese}.
So this normalization is the same as the one used in the complex
setting by Yoccoz in~\cite{Yoc95b}.
In what follows we only consider normalized power series.

A normalized power series of the form~\eqref{e:germ} maps~$\mK$ to
itself isometrically, see for example~\cite[\S$1.3$]{Rivthese}.
If either the residue characteristic of~$K$ is zero or the
reduction~$\tlambda$ of~$\lambda$ has infinite order in~$\tK^*$, then
a normalized power series as in~\eqref{e:germ} has at most a finite number of periodic points, see Lemma~\ref{l:admissible periods}.
Thus, to simplify the exposition, in the rest of this introduction we assume that the residue characteristic~$p$ of~$K$ is positive and that the order~$q$ of~$\tlambda$ in~$\tK^*$ is finite.
Then~$q$ is not divisible by~$p$, and the minimal period of each periodic point in~$\mK \setminus \{ 0 \}$ is of the form~$qp^n$, for some integer~$n \ge 0$, see Lemma~\ref{l:admissible periods}.

\subsection{Periodic points lower bound}
\label{ss:periodic bound intro}
The main theme of this paper is the optimality of the following lower bound.
\begin{generic}[Periodic Points Lower Bound]
Let~$p$ be a prime number, and~$(K, |\cdot|)$ an ultrametric field of residue characteristic~$p$.
Moreover, let~$\lambda$ in~$K$ be such that~$|\lambda| = 1$, and such that the order~$q$ of~$\tlambda$ in~$\tK^*$ is finite.
Then, for every power series
$$ f(z) = \lambda z + \cdots  \text{ in } \OK[[z]] $$
and every integer~$n \ge 1$ such that~$\lambda^{qp^n} \neq 1$, the following property holds: For every periodic point~$z_0$ of~$f$ of minimal period~$qp^n$, we have
\begin{equation}
\label{e:periodic bound}
|z_0|
\ge
\left| \frac{\lambda^{q p^n} - 1}{\lambda^{q p^{n - 1}} - 1} \right|^{\frac{1}{q p^n}}.
\end{equation}
\end{generic}
See Lemma~\ref{l:periodic bound} for a more detailed statement, which includes a lower bound for periodic points of minimal period~$q$.

Although the statement of the Periodic Points Lower Bound is new, similar estimates were shown in~\cite{EinEveWar04a,EinEveWar04b} in the case~$K$ is of characteristic zero.
The idea of the proof can be traced back, at least, to Cremer's
example of a complex polynomial having an irrationally indifferent
fixed point that is not locally linearizable, see for example~\cite[\S$11$]{Mil06}.
It boils down to the observation that the product of the norms of all the fixed points of~$f^{qp^n}$ in~$\mK \setminus \{ 0 \}$ is equal to~$\left| \lambda^{qp^n} - 1 \right|$.

Suppose that the characteristic of~$K$ is equal to~$p$ and that~$\lambda^q \neq 1$.
Then the lower bound in~\eqref{e:periodic bound} is equal to~$\left| \lambda^q - 1 \right|^{\frac{p - 1}{qp}}$, which is independent of~$n$.
Thus, the following corollary is a direct consequence of the Periodic Points Lower Bound.
\begin{coro}
\label{c:irrational are isolated}
Every irrationally indifferent periodic point is isolated in positive characteristic.
\end{coro}
Combined with the fact that the quadratic polynomial~$\lambda z + z^2$ in~$K[z]$ is not locally linearizable at~$z = 0$ when~$p$ is odd and~$|\lambda - 1| < 1$, see~\cite[Theorem~$2.3$]{Lin04},\footnote{According to Herman's conjecture~\cite[Conjecture~$2$]{Her87}, in positive characteristic a typical indifferent periodic point is not locally linearizable; yet, it is isolated as a periodic point by Corollary~\ref{c:irrational are isolated}. P{\'e}rez-Marco showed that in the complex setting there are maps with similar properties, see~\cite[Theorem~I.$3.1$]{Per97}.} the corollary above shows that in odd characteristic the existence of small cycles is not an obstruction to local linearization,\footnote{In a field of characteristic~$2$, the quadratic polynomial~$\lambda z + z^2$ is locally linearizable at~$z = 0$ when~$|\lambda - 1| < 1$, see~\cite[Theorem~$2.3$]{Lin04}. We can consider instead the polynomial~$\lambda z + z^3$, which is not locally linearizable at~$z = 0$ when~$|\lambda - 1| < 1$, see~\cite[Theorem~$1.1$]{Lin10}. Thus, Corollary~\ref{c:irrational are isolated} also proves that in characteristic~$2$ the existence of small cycles is not an obstruction to local linearization.} see~\cite[Corollary~C]{Lin13} for a somewhat analogous phenomenon in the~$p$-adic setting.
This is in contrast with the complex field case: Yoccoz showed that if~$\lambda$ in~$\C^*$ is not a root of unity and the quadratic polynomial~$\lambda z + z^2$ in~$\C[z]$ is not locally linearizable at~$z = 0$, then every neighborhood of~$z = 0$ contains a periodic cycle, see~\cite[\S$6.6$]{Yoc95b}.

In view of Corollary~\ref{c:irrational are isolated}, we propose the following conjecture.
\begin{conj}
In positive characteristic, every periodic point whose multiplier is a root of unity is either isolated as a periodic point, or has a neighborhood on which an iterate of the map is the identity.
\end{conj}

In~\cite{LinRivparabolic} we solve this conjecture in the affirmative,
in the case of generic parabolic points.
For an ultrametric ground field of characteristic zero, the assertion of the conjecture does hold: When the residue characteristic is zero it follows from Lemma~\ref{l:admissible periods}, and when the residue characteristic is positive it follows from the fact that periodic points are the zeros of the iterative logarithm, see~\cite[Proposition~$3.6$]{Rivthese} and also~\cite{Lub94} for the case where~$K$ is discretely valued.

Suppose now that~$K$ is of characteristic zero and that~$\lambda$ is not a root of unity.
Then a direct computation shows that the lower bound in~\eqref{e:periodic bound} converges to~$1$ as~$n \to + \infty$.
So, the Periodic Points Lower Bound implies that for every~$r$ in~$(0, 1)$ the number of periodic points of~$f$ in~$\{ z \in K : |z| \le r \}$ is finite.\footnote{This also follows from the fact that the periodic points of~$f$ are the zeros of the iterative logarithm of~$f$, which is given by an analytic power series that converges on~$\mK$, see~\cite[Proposition~$3.16$]{Rivthese} and also~\cite{Lub94} for the case where~$K$ is discretely valued.}
In fact, the Periodic Points Lower Bound gives a quantitative estimate
of the speed at which periodic points separate from~$z = 0$ as the
period increases: Just observe that there is~$n_0 \ge 1$ that only depends on~$\lambda^q$, such that for every integer~$n \ge n_0$ we have
$$ \left| \frac{\lambda^{q p^n} - 1}{\lambda^{q p^{n - 1}} - 1} \right|
=
|p|, $$
see also~\cite[Remark~$3.6$]{EinEveWar04a} and~\cite[Theorem~$1$]{EinEveWar04b}.

\subsection{Optimal cycles}
\label{ss:optimal cycles}
Let~$p$, $K$, $\lambda$, and $q$ be as in the Periodic Points Lower Bound.
Then we say that a power series~$f(z) = \lambda z + \cdots$ in~$\OK[[z]]$ \emph{has an optimal cycle of period~$qp^n$}, if~$f$ has a periodic point~$z_0$ of minimal period~$qp^n$ such that~\eqref{e:periodic bound} holds with equality.
If~$f$ has an optimal cycle of period~$qp^n$, then this is in fact the only cycle of minimal period~$qp^n$ of~$f$, see Theorem~\ref{t:optimality and minimality} below.

\begin{theoalph}
\label{t:periodic optimality}
Let~$p$ be a prime number, and~$(K, |\cdot|)$ an algebraically closed ultrametric field of residue characteristic~$p$.
Moreover, let~$\lambda$ in~$K$ be such that~$|\lambda| = 1$, such that the order~$q$ of~$\tlambda$ in~$\tK^*$ is finite, and such that~$\lambda^q \neq 1$.
If~$K$ is of characteristic zero, assume in addition that~$\lambda$ is transcendental over the prime field of~$K$.
Then there is a polynomial~$P(z) = \lambda z + \cdots$ in~$\OK[z]$ of degree at most~$2q + 1$, having for each integer~$n \ge 1$ an optimal cycle of period~$qp^n$.
\end{theoalph}

We use some explicit polynomials to prove this theorem, see Propositions~\ref{p:odd optimality} and~\ref{p:even periodic optimality} in~\S\ref{ss:concrete optimality} for details.
For example, in the case~$p$ is odd and~$q = 1$, we prove that the polynomial~$\lambda z + z^2$ satisfies the conclusions of Theorem~\ref{t:periodic optimality}, in agreement with the situation in the complex setting, see~\cite[\S$6.6$]{Yoc95b}.
However, not every quadratic polynomial has this property: If~$p = 11$ and~$\tlambda = -1$, then there is no integer~$n \ge 1$ for which the quadratic polynomial~$\lambda z + z^2$ has an optimal cycle of period~$2p^n$, see~\S\ref{ss:quadratic optimal cycles}.

Suppose~$K$ is of characteristic~$p$, and let~$P$ be a polynomial satisfying the conclusions of Theorem~\ref{t:periodic optimality}.
Then all the periodic points of~$P$ of minimal period at least~$qp$ are in fact concentrated in the sphere
$$ \left\{ z \in K : |z| = |\lambda^q - 1|^{\frac{p - 1}{qp}} \right\}.\footnote{As pointed out in~\S\ref{ss:periodic bound intro}, such a concentration of periodic points cannot occur in the case where the characteristic of~$K$ is zero.}$$
It is not clear to us how the periodic points are distributed in this sphere.
For concreteness, we propose the following problem.
\begin{prob}
Let~$p$ be an odd prime number, $(K, |\cdot|)$ an algebraically closed and complete ultrametric field of characteristic~$p$, and~$\lambda$ in~$K$ such that~$0 < |\lambda - 1| < 1$.
Moreover, for each integer~$n \ge 1$ let~$\Pi_n$ be the set of
cardinality~$p^n$ of all
periodic points of~$P_{\lambda}(z) = \lambda z + z^2$ in~$\mK$ of minimal period~$p^n$.
Is the sequence of measures
$$ \left\{ \frac{1}{p^n} \sum_{z \in \Pi_n} \delta_z \right\}_{n = 1}^{+ \infty}$$
convergent?
\end{prob}
It would be natural to consider the measures above as measures on the
Berkovich projective line of~$K$, since they would accumulate on at
least one measure with respect to the corresponding weak* topology, see~\cite{Ber90}.
Furthermore, every accumulation measure would be invariant by the
action of~$P_{\lambda}$ on the Berkovich projective line of~$K$.

\subsection{Minimally ramified power series}
\label{ss:minimally ramified intro}
One of the main ingredients in the proof of Theorem~\ref{t:periodic optimality} is (an extension of) the concept of ``minimally ramified'' power series introduced by Laubie, Movahhedi, and Salinier in~\cite{LauMovSal02}, in their study of Lubin's conjecture in~\cite{Lub94}.
To introduce this concept, let~$p$ be a prime number, and~$k$ a field of characteristic~$p$.
Denote by~$\ord(\cdot)$ the valuation on~$k[[\zeta]]$ defined for a nonzero power series as the lowest degree of its nonzero terms, and for the zero power series~$0$ by~$\ord(0) = + \infty$.
For a power series of the form~$g(\zeta) = \zeta + \cdots$ in~$k[[\zeta]]$, define for each integer~$n \ge 0$ the number
$$ i_n(g) \= \ord \left( \frac{g^{p^n}(\zeta) - \zeta}{\zeta} \right). $$
As observed in~\cite{LauMovSal02}, the results of Sen in~\cite{Sen69} imply that for every integer~$n \ge 0$ we have~$i_n(g) \ge \frac{p^{n + 1} - 1}{p - 1}$; following~\cite{LauMovSal02}, the power series~$g$ is called \emph{minimally ramified} if the equality holds for every~$n$.\footnote{We note that in the case~$p = 2$, a minimally ramified power series in the sense of~\cite{LauMovSal02} is what we call here an ``almost minimally ramified'' power series, see~\S\ref{ss:almost minimally ramified} for precisions.}

To prove Theorem~\ref{t:periodic optimality}, we need to deal with a more general class of power series, allowing~$g'(0)$ to be an arbitrary root of unity.
For this, we prove a higher order version of the main theorem of Sen in~\cite{Sen69}, see Theorem~\ref{t:higher order sen} in~\S\ref{ss:higher order sen}.
We use it to show that for every integer~$q \ge 1$ not divisible by~$p$, every root of unity~$\gamma$ in~$k$ of order~$q$, and every power series of the form
$$ g(\zeta) = \gamma \zeta + \cdots
\text{ in } k[[\zeta]], $$
we have for every integer~$n \ge 0$
\begin{equation}
  \label{e:minimally ramified at n intro}
i_n(g^q) \ge q \frac{p^{n + 1} - 1}{p - 1},  
\end{equation}
see Proposition~\ref{p:minimally ramified} in~\S\ref{ss:minimally ramified}.
We say that~$g$ is \emph{minimally ramified} if equality holds for
every~$n$, see Definition~\ref{ss:minimally ramified}.
We give a characterization of minimally ramified power series in terms
of the iterative residue, see Theorem~\ref{t:characterization of
  minimally ramified} in~\S\ref{s:characterization of minimally ramified}.

The following links the existence of optimal cycles to minimally ramified maps.
To simplify the exposition we have restricted to ground fields of odd residue characteristic.
An analogous statement holds for ground fields of residue characteristic~$2$, see Theorem~\ref{t:optimality and minimality}' in~\S\ref{ss:generic optimality}.

\begin{theoalph}
\label{t:optimality and minimality}
Let~$p$ be an odd prime number, $(K, |\cdot|)$ an algebraically closed field of residue characteristic~$p$, and~$q \ge 1$ an integer that is not divisible by~$p$.
Then the following properties hold.
\begin{enumerate}
\item[1.]
Let~$\lambda$ in~$K$ be such that~$|\lambda| = 1$ and such that the order of~$\tlambda$ in~$\tK^*$ is~$q$.
Moreover, let $n \ge 1$ be an integer and~$P(z) = \lambda z + \cdots $ a polynomial in~$\OK[z]$ having an optimal cycle of period~$qp^n$.
Then this is the only cycle of minimal period~$qp^n$ of~$P$, and the reduction of~$P$ is minimally ramified.
\item[2.]
Let~$\gamma$ be a root of unity in~$\tK$ of order~$q$, and~$g(\zeta) = \gamma \zeta + \cdots$ a polynomial in~$\tK[\zeta]$ that is minimally ramified.
Given an integer~$d \ge \max \{ \deg(g), p \}$, let~$a_1$, \ldots, $a_d$ in~$\OK$ be algebraically independent over the prime field of~$K$, and such that the reduction of the polynomial~$P(z) \= a_1 z + \cdots + a_d z^d$ is~$g$.
Then for every integer~$n \ge 1$, the polynomial~$P$ has a unique cycle of minimal period~$qp^n$, and this cycle is an optimal cycle of period~$qp^n$ of~$P$.
\end{enumerate}
\end{theoalph}

See Proposition~\ref{p:detailed optimality criterion} for a related result that holds under a weaker form of the genericity condition in part~$2$.
This genericity condition is necessary to prevent the concentration of periodic points, as it occurs, for example, for the polynomials studied in Appendix~\ref{s:all minimal}.

\subsection{Optimal cycles of quadratic polynomials}
\label{ss:quadratic optimal cycles}
The following is a more precise version of Theorem~\ref{t:optimality and minimality} for quadratic polynomials.
\begin{theoalph}
\label{t:quadratic optimality and minimality}
Let~$p$ be an odd prime number, and~$(K, |\cdot|)$ an algebraically closed ultrametric field of residue characteristic~$p$.
Moreover, let~$\lambda$ in~$K$ be such that~$|\lambda| = 1$, such that the order~$q$ of~$\tlambda$ in~$\tK^*$ is finite, and such that~$\lambda^q \neq 1$.
If~$K$ is of characteristic zero, assume in addition that~$\lambda$ is transcendental over the prime field of~$K$.
Then the following dichotomy holds for the polynomial~$P_{\lambda}(z) \= \lambda z + z^2$:
\begin{enumerate}
\item[1.]
If the reduction of~$P_{\lambda}$ is minimally ramified, then for each integer~$n \ge 1$ the polynomial~$P_{\lambda}$ has a unique cycle of minimal period~$qp^n$, and this cycle is an optimal cycle of period~$qp^n$ of~$P_{\lambda}$;
\item[2.]
If the reduction of~$P_{\lambda}$ is not minimally ramified, then there is no integer~$n \ge 1$ for which the polynomial~$P_{\lambda}$ has an optimal cycle of period~$qp^n$.
\end{enumerate}
\end{theoalph}

The first alternative of the theorem always holds when~$q = 1$,
because the polynomial~$\zeta + \zeta^2$ in~$\tK[\zeta]$ is minimally
ramified, see~\cite[\emph{Exemple}~$3.19$]{Rivthese} or
Proposition~\ref{p:concrete minimally ramified} in~\S\ref{ss:proof of characterization of minimally ramified}.
For an example where the second alternative holds, suppose~$p = 11$ and consider the quadratic polynomial
$$ g_0(\zeta) \= - \zeta + \zeta^2 \text{ in } \tK[\zeta]; $$
a direct computation shows that~$i_0(g_0^2) = 2$ and~$i_1(g_0^2) > 24$, so~$g_0$ is not minimally ramified.
So, when~$p = 11$ and~$\tlambda = -1$, the second alternative of Theorem~\ref{t:quadratic optimality and minimality} holds.

\begin{prob}
Let~$p$ be an odd prime number, $\F_p$ a field of~$p$ elements, and~$\overline{\F}_p$ an algebraic closure of~$\F_p$.
Determine all those~$\gamma$ in~$\overline{\F}_p^*$ for which the quadratic polynomial~$\gamma \zeta + \zeta^2$ in~$\overline{\F}_p[\zeta]$ is minimally ramified.
\end{prob}

To prove that for~$\gamma$ in~$\overline{\F}_p^*$ the polynomial~$\gamma \zeta + \zeta^2$ is minimally ramified, it is enough to show that~\eqref{e:minimally ramified at n intro} holds with equality with~$g(\zeta) = \gamma \zeta + \zeta^2$ and~$n = 1$, see Proposition~\ref{p:minimally ramified} in~\S\ref{ss:minimally ramified}.

We note that for~$\gamma$ in~$\overline{\F}_p$, the property of~$\gamma \zeta + \zeta^2$ being minimally ramified does not depend on the order of~$\gamma$ alone.
For example, when~$p = 7$, the orders of~$2$ and~$4$ in~$\overline{\F}_7^*$ are both equal to~$3$, but~$2 \zeta + \zeta^2$ is not minimally ramified, and~$4 \zeta + \zeta^2$ is.
\subsection{Organization}
\label{ss:organization}
In~\S\ref{s:periodic bound} we give general properties of normalized power series.
After some preliminaries in~\S\ref{ss:preliminaries}, in~\S\ref{ss:admissible periods} we describe the minimal periods of cycles of a normalized power series (Lemma~\ref{l:admissible periods}).
In~\S\ref{ss:periodic bound} we prove a more general version of the Periodic Points Lower Bound (Lemma~\ref{l:periodic bound}).

In~\S\ref{s:minimally ramified} we introduce and study minimally ramified power series.
We start by proving a higher order version of Sen's theorem in~\S\ref{ss:higher order sen}.
In~\S\ref{ss:minimally ramified} we use this result to define and characterize minimally ramified power series.
In~\S\ref{ss:almost minimally ramified} we study a variant of this concept for fields of characteristic~$2$.
In our characterization of minimally ramified power series we use an extension of Laubie and Sa{\"{\i}}ne~\cite{LauSai98} of a result of Keating~\cite{Kea92}.

In~\S\ref{s:characterization of minimally ramified} we characterize,
for each integer~$q$ not divisible by~$p$ and each root of
unity~$\gamma$ of order~$q$, those power series of the form $g(\zeta)
= \gamma \zeta + \cdots$ that are minimally ramified in terms of the
iterative residue of~$g$ (Theorem~\ref{t:characterization of minimally ramified}).
A direct consequence is that there is a minimally ramified polynomial~$g$ as above of degree~$q + 1$ or~$2q + 1$.

In~\S\ref{s:optimality} we prove Theorems~\ref{t:periodic optimality}, \ref{t:optimality and minimality}, and~\ref{t:quadratic optimality and minimality}.
In~\S\ref{ss:generic optimality} we prove a general version of Theorem~\ref{t:optimality and minimality}, that we state as Theorem~\ref{t:optimality and minimality}'.
In~\S\ref{ss:concrete optimality} we exhibit concrete polynomials that satisfy the conclusions of Theorem~\ref{t:periodic optimality} (see Propositions~\ref{p:odd optimality} and~\ref{p:even periodic optimality}).
In Appendix~\ref{s:all minimal} we study a concentration of periodic points phenomenon showing that a very natural candidate to have optimal cycles in characteristic~$2$ has none.
The proof of Theorem~\ref{t:quadratic optimality and minimality} is given at the end of~\S\ref{ss:concrete optimality}.

\subsection{Acknowledgments}
We thank Tom Tucker for useful conversations regarding the lifting
argument in Lemmas~\ref{l:generic independence zero characteristic}
and~\ref{l:generic independence}.
We also thank Thomas Ward and the referees for various comments that
helped improve the exposition of the paper.
Part of this article was developed while the authors were visiting the Institute for Computational and Experimental Research in Mathematics (ICERM).
We would like to thank this institute for the excellent working conditions provided.

\section{Periodic points lower bound}
\label{s:periodic bound}
The purpose of this section is to prove general facts about periodic points of normalized power series.
After some preliminaries in~\S\ref{ss:preliminaries}, in~\S\ref{ss:admissible periods} we describe the minimal periods of periodic points of normalized power series.
In~\S\ref{ss:periodic bound} we prove a general version of the Periodic Points Lower Bound, stated in~\S\ref{ss:periodic bound intro}, that we state as Lemma~\ref{l:periodic bound}.

\subsection{Preliminaries}
\label{ss:preliminaries}

Given a ring~$R$ and an element~$a$ of~$R$, we denote by~$\left\langle a \right\rangle$ the ideal of~$R$ generated by~$a$.

Given a field~$k$, denote by~$k^* \= k \setminus \{ 0 \}$ the multiplicative subgroup of~$k$.
A nonzero element~$\gamma$ of~$k^*$ has \emph{infinite order in~$k^*$}, if for every integer~$q \ge 1$ we have~$\gamma^q \neq 1$.
If~$\gamma$ is not of infinite order in~$k^*$, then the \emph{order of~$\gamma$ in~$k^*$} is the least integer~$q \ge 1$ such that~$\gamma^q = 1$.
When~$k$ is of positive characteristic, in this last case the order~$\gamma$ is not divisible by the characteristic of~$k$.

Let~$p$ be a prime number, and~$k$ a field of characteristic~$p$.
The \emph{order} of a nonzero power series~$g(\zeta)$ in~$k[[\zeta]]$ is the lowest degree of a nonzero term in~$g(\zeta)$.
The order of the zero power series in~$k[[\zeta]]$ is~$+ \infty$.
For~$g(\zeta)$ in~$k[[\zeta]]$, denote by~$\ord(g)$ the order of~$g$.
The function~$\ord$ so defined is a valuation on~$k[[\zeta]]$.

Let~$(K, | \cdot |)$ be an ultrametric field.
Denote by~$\OK$ the ring of integers of~$K$, by~$\mK$ the maximal ideal of~$\OK$, and by~$\tK \= \OK / \mK$ the residue field of~$K$.
Moreover, denote the projection in~$\tK$ of an element~$a$ of~$\OK$ by~$\widetilde{a}$; it is the \emph{reduction of~$a$}.
The \emph{reduction} of a power series~$f(z)$ in~$\OK[[z]]$ is the power series in~$\tK[[\zeta]]$ whose coefficients are the reductions of the corresponding coefficients of~$f$.

For such a power series~$f(z)$ in~$\OK[[z]]$, the \emph{Weierstrass degree~$\wideg(f)$ of~$f$} is the order in~$\tK[[\zeta]]$ of the reduction~$\wtf(\zeta)$ of~$f(z)$.
When~$K$ is algebraically closed and~$\wideg(f)$ is finite, $\wideg(f)$ is equal to the number
of zeros of~$f$ in~$\mK$, counted with multiplicity; see for example~\cite[\S VI, Theorem~$9.2$]{Lan02}.

Notice that a power series~$f(z)$ in~$\OK[[z]]$ converges in~$\mK$.
If in addition~$|f(0)| < 1$, then by the ultrametric inequality~$f$ maps~$\mK$ to itself.
In this case, a point~$z_0$ in~$\mK$ is \emph{periodic for~$f$}, if there is an integer~$n \ge 1$ such that~$f^n(z_0) = z_0$.
In this case \emph{$z_0$ is of period~$n$}, and~$n$ is a \emph{period of~$z_0$}.
If in addition~$n$ is the least integer with this property, then~$n$ is the \emph{minimal period of~$z_0$} and~$(f^n)'(z_0)$ is the \emph{multiplier of~$z_0$}.
Note that an integer~$n \ge 1$ is a period of~$z_0$ if and only if it is divisible by the minimal period of~$z_0$.
Given a periodic point~$z_0$ of~$f$ of multiplier~$\lambda$, we say~$z_0$ is \emph{attracting} if~$|\lambda| < 1$, \emph{indifferent} if~$|\lambda| = 1$, and \emph{repelling} if~$|\lambda| > 1$.
In the case~$z_0$ is indifferent, $z_0$ is \emph{rationally indifferent} or~\emph{parabolic} if~$\lambda$ is a root of unity, and it is \emph{irrationally indifferent} otherwise.

\subsection{Minimal periods of normalized power series}
\label{ss:admissible periods}
The purpose of this section is to prove the following lemma, where we gather well-known results on periodic points of a normalized power series.

\begin{lemm}
\label{l:admissible periods}
Let~$(K, | \cdot |)$ be an ultrametric field and~$\lambda$ in~$K$ such that~$|\lambda| = 1$.
Then for every power series~$f(z) = \lambda z + \cdots $ in~$\OK[[z]]$, the following properties hold.
\begin{enumerate}
\item[1.]
If~$r \ge 1$ is an integer such that~$|\lambda^r - 1| = 1$, then~$f$ has no periodic point of period~$r$, other than~$z = 0$.
In particular, if~$\tlambda$ has infinite order in~$\tK^*$, then~$f$ has no periodic point other than~$z = 0$.
\item[2.]
Suppose the order~$q$ of~$\tlambda$ in~$\tK^*$ is finite.
\begin{enumerate}
\item
If the residue characteristic of~$K$ is zero, then the minimal period of each periodic point of~$f$ in~$\mK \setminus \{ 0 \}$ is equal to~$q$.
\item 
If the residue characteristic~$p$ of~$K$ is positive, then~$p$ does not divide~$q$ and the minimal period of each periodic point of~$f$ in~$\mK \setminus \{ 0 \}$ is of the form~$q p^n$, for some integer~$n \ge 0$.
\end{enumerate}
\end{enumerate}
\end{lemm}
The proof of this lemma is after the following one.
\begin{lemm}
\label{l:fixed are periodic}
Let~$(K, | \cdot |)$ be a complete ultrametric field and~$g(z)$ a power series in~$\OK[[z]]$ such that~$|g(0)| < 1$.
Then for each integer~$m \ge 1$ the power series~$g(z) - z$ divides~$g^m(z) - z$ in~$\OK[[z]]$.
\end{lemm}
\begin{proof}
We proceed by induction in~$m$, the case~$m = 1$ being trivial.
Let~$m \ge 1$ be an integer for which the lemma holds.
Note that it is enough to show that~$g(z) - z$ divides~$g^{m + 1}(z) - g^m(z)$ in~$\OK[[z]]$.
Writing~$g^m(z) = \sum_{n = 0}^{+ \infty} a_n z^n$ and using~$|g(0)| < 1$, we have that~$\sum_{n = 0}^{+ \infty} a_n (g(z)^n - z^n)$ converges to a power series in~$\OK[[z]]$ and that
\begin{equation}
  \label{e:cocycle}
  \sum_{n = 0}^{+ \infty} a_n (g(z)^n - z^n) = g^{m + 1}(z) - g^m(z).
\end{equation}
On the other hand, using again~$|g(0)| < 1$, we have that the series
$$ \sum_{n = 0}^{+ \infty} a_n \sum_{j = 0}^{n - 1} z^j g(z)^{n - 1 - j} $$
converges to a power series~$h(z)$ in~$\OK[[z]]$, and that~$h(z) (g(z) - z)$ is equal to~\eqref{e:cocycle}.
This completes the proof of the lemma.
\end{proof}
\begin{proof}[Proof of Lemma~\ref{l:admissible periods}]
To prove part~$1$, let~$r \ge 1$ be an integer such that~$|\lambda^r - 1| = 1$, and note that the constant term of the power series~$(f^r(z) - z) / z$ is equal to~$\lambda^r - 1$.
Thus, by the ultrametric inequality, for each~$z_0$ in~$\mK \setminus \{ 0 \}$ we have
$$ |f^r(z_0) - z_0|
=
|\lambda^r - 1| \cdot |z_0|
=
|z_0|
\neq 0. $$
Thus~$f^r(z_0) \neq z_0$ and therefore~$z_0$ is not a periodic point of period~$r$ of~$f$.
This proves part~$1$.

To prove part~$2$, suppose first the residue characteristic~$p$ of~$K$ is positive.
We prove first that~$q$ is not divisible by~$p$.
Suppose by contradiction that~$q$ is divisible by~$p$, and put~$m = q / p$.
By the minimality of~$q$ we have~$\tlambda^{m} \neq \widetilde{1}$.
On the other hand, $(\tlambda^{m})^p = \tlambda^{q} = \widetilde{1}$.
Since the characteristic of~$\tK$ is equal to~$p$, this implies that~$\tlambda^m = \widetilde{1}$.
This contradiction proves that~$q$ is not divisible by~$p$.

To complete the proof of part~$2$, we prove simultaneously part~$2$(a) and the second assertion of part~$2$(b).
To do this, let~$\ell \ge 1$ be an integer and~$z_0$ a periodic point of~$f$ in~$\mK \setminus \{ 0 \}$ of minimal period~$\ell$.
By part~$1$ we must have~$|\lambda^\ell - 1| < 1$, or equivalently~$\tlambda^\ell = \widetilde{1}$.
Thus~$q$ divides~$\ell$.
If the residue characteristic of~$K$ is zero, put~$q_0 \= q$.
If the residue characteristic~$p$ of~$K$ is positive, let~$n \ge 0$ be the largest integer such that~$p^n$ divides~$\ell$ and put~$q_0 \= qp^n$.
In both cases we have that~$q_0$ divides~$\ell$.
To complete the proof of part~$2$, it is enough to prove that~$\ell = q_0$.
Suppose by contradiction that~$\ell$ is not equal to~$q_0$, so that~$m \= \ell / q_0 \ge 2$.
Then, by Lemma~\ref{l:fixed are periodic} with~$g = f^{q_0}$, the power series~$f^{q_0}(z) - z$ divides~$f^{\ell}(z) - z$ in~$\OK[[z]]$.
Note that~$z_0$ is a zero of the power series~$(f^{\ell}(z) - z) / (f^{q_0}(z) - z)$.
However, if~$\lambda^{q_0} \neq 1$, then the constant term of this power series is equal to
$$ \frac{\lambda^{\ell} - 1}{\lambda^{q_0} - 1}
=
1 + \lambda^{q_0} + \cdots + \lambda^{(m - 1) q_0}, $$
whose norm equal to~$1$; so the power series~$(f^{\ell}(z) - z) / (f^{q_0}(z) - z)$ does not have zeros in~$\mK$.
We thus obtain a contradiction that completes the proof of part~$2$ when~$\lambda^{q} \neq 1$.
It remains to consider the case where~$\lambda^q = 1$.
In this case the order of~$f^{q_0}(z) - z$ in~$K[[z]]$ is at least~$2$.
If the order of this power series is infinite, then~$f^{q_0}(z) = z$ and therefore every point of~$\mK$ would be periodic of period~$q_0$; but this is not possible because~$z_0$ is periodic of minimal period~$\ell$, and by assumption~$\ell > q_0$.
This proves that the order~$t$ of the power series~$f^{q_0}(z) - z$ is finite and at least~$2$.
If we denote by~$a$ the coefficient of~$z^t$ in~$f^{q_0}(z)$, then a straightforward induction argument shows that for every integer~$s \ge 1$ we have
$$ f^{s q_0}(z) = z + s a z^t + \cdots . $$
When~$s = m$, we obtain
$$ f^{\ell}(z) = z + m a z^t + \cdots . $$
This implies that the constant term of the power series~$(f^{\ell}(z) - z) / (f^{q_0}(z) - z)$ is equal to~$m$, which has norm~$1$.
As before, this implies that this power series has no zeros in~$\mK$, and we obtain a contradiction that completes the proof of the lemma.
\end{proof}

\subsection{Periodic points lower bound}
\label{ss:periodic bound}
The purpose of this section is to give, for a normalized power series with an irrationally indifferent fixed point at~$z = 0$, a lower bound for the norms of periodic points different from~$z = 0$.
The bound depends only on the multiplier of the fixed point~$z = 0$, and of the minimal period of the periodic point.

\begin{lemm}
\label{l:periodic bound}
Let~$p$ be a prime number, and $(K, | \cdot |)$ an ultrametric field of residue characteristic~$p$.
Let~$\lambda$ in~$K$ be such that~$|\lambda| =  1$, and such that the order~$q$ of~$\tlambda$ in~$\tK^*$ is finite.
Then for every power series~$f(z) = \lambda z + \cdots $ in~$\OK[[z]]$, the following properties hold.
\begin{enumerate}
\item[1.]
Suppose~$\lambda^q \neq 1$, and let~$w_0$ be a periodic point of~$f$ of minimal period~$q$.
In the case~$q = 1$, assume~$w_0 \neq 0$.
Then we have
\begin{equation}
\label{e:fixed bound}
|w_0| \ge |\lambda^q - 1|^{\frac{1}{q}},
\end{equation}
with equality if and only if~$\wideg \left( f^q(z) - z \right) = q + 1$.
Moreover, if equality holds, then the cycle containing~$w_0$ is the only cycle of minimal period~$q$ of~$f$ in~$\mK \setminus \{ 0 \}$, and for every point~$w_0'$ in this cycle the inequality above holds with equality with~$w_0$ replaced by~$w_0'$.
\item[2.]
Let~$n \ge 1$ be an integer such that~$\lambda^{q p^n} \neq 1$, and~$z_0$ a periodic point of~$f$ of minimal period~$q p^n$.
Then we have
\begin{equation}
\label{e:periodic bound bis}
|z_0|
\ge
\left| \frac{\lambda^{q p^n} - 1}{\lambda^{q p^{n - 1}} - 1} \right|^{\frac{1}{q p^n}},
\end{equation}
with equality if and only if
\begin{equation}
\label{e:minimal wideg increment}
\wideg \left( \frac{f^{q p^n}(z) - z}{f^{q p^{n - 1}}(z) - z} \right)
=
q p^n.
\end{equation}
Moreover, if equality holds, then the cycle containing~$z_0$ is the only cycle of minimal period~$qp^n$ of~$f$, and for every point~$z_0'$ in this cycle the inequality above holds with equality with~$z_0$ replaced by~$z_0'$.
\end{enumerate}
\end{lemm}

Note that in~\eqref{e:minimal wideg increment} above we use the fact that~$f^{qp^{n - 1}}(z) - z$ divides~$f^{q p^n}(z) - z$ in~$\OK[[z]]$, given by Lemma~\ref{l:fixed are periodic} with~$g = f^{q p^{n - 1}}$ and~$m = p$.

The proof of Lemma~\ref{l:periodic bound} is below, after the following lemma.
\begin{lemm}
  \label{l:elementary division}
Let~$K$ be a complete ultrametric field and let~$h(z)$ be a power series in~$\OK[[z]]$.
If~$\xi$ is a zero of~$h$ in~$\mK$, then~$z - \xi$ divides~$h(z)$ in~$\OK[[z]]$.
\end{lemm}
\begin{proof}
Put~$T(z) = z + \xi$ and note that~$h \circ T(z)$ vanishes at~$z = 0$ and is in~$\OK[[z]]$.
This implies that~$z$ divides~$h \circ T(z)$ in~$\OK[[z]]$.
Letting~$g(z) \= h \circ T(z) / z$, it follows that the power series~$g \circ T^{-1}(z) = h(z) / (z - \xi)$ is in~$\OK[[z]]$, as wanted.
\end{proof}

\begin{proof}[Proof of Lemma~\ref{l:periodic bound}]
Replacing~$K$ by one of its completions if necessary, assume~$K$ complete.

We use the fact that, since~$|f'(0)| = 1$, the power series~$f$ maps~$\mK$ to itself isometrically, see for example~\cite[\S$1.3$]{Rivthese}.

\partn{1}
To prove~\eqref{e:fixed bound}, let~$w_0$ in~$\mK \setminus \{ 0 \}$ be a periodic point of~$f$ of
minimal period~$q$.
Note that every point in the forward orbit~$\cO$ of~$w_0$ under~$f$ is a zero of the power series~$(f^q(z) - z) / z$, and that the constant term of this power series is~$\lambda^q - 1$.
On the other hand, $\cO$ consists of~$q$ points, and, since~$f$ maps~$\mK$ to itself isometrically, all the points in~$\cO$ have the same norm.
Applying Lemma~\ref{l:elementary division} inductively with~$\xi$ replaced by each element of~$\cO$, it follows that~$\prod_{w_0' \in \cO} (z - w_0')$ divides~$(f^q(z) - z)/z$ in~$\OK[[z]]$.
In particular, the constant term
\begin{equation}
\label{e:total product of fixed points}
\left( \lambda^q - 1 \right) / \prod_{w_0' \in \cO} (- w_0')
\end{equation}
of the power series $((f^q(z) - z)/z) / \prod_{w_0' \in \cO} (z - w_0')$ is in~$\OK$.
We thus have
\begin{equation}
\label{e:powered fixed lower bound}
|w_0|^q = \prod_{w_0' \in \cO} |w_0'|
\ge
\left| \lambda^q - 1 \right|,
\end{equation}
and therefore~\eqref{e:fixed bound}.
Moreover, equality holds precisely when the constant term~\eqref{e:total product of fixed points} of~$((f^q(z) - z)/z) / \prod_{w_0' \in \cO} (z - w_0')$ has norm equal to~$1$.
Equivalently, equality in~\eqref{e:powered fixed lower bound} holds if and only if~$\wideg(f^q(z) - z) = q + 1$.
Finally, when this last equality holds, the set~$\cO$ is the set of all zeros of~$(f^q(z) - z)/z$ in~$\mK$, so~$\cO$ is the only cycle of minimal period~$q$ of~$f$.
This completes the proof of part~$1$.

\partn{2}
To prove~\eqref{e:periodic bound bis}, let~$n \ge 1$ be an integer such
that~$\lambda^{q p^n} \neq 1$, and~$z_0$ a periodic point of~$f$ of minimal period~$qp^n$.
By Lemma~\ref{l:fixed are periodic} with~$g = f^{q p^{n - 1}}$ and~$m = p$, the power series~$f^{q p^{n - 1}}(z) - z$ divides~$f^{q p^n}(z) - z$ in~$\OK[[z]]$.
Note that every point in the forward orbit~$\cO$ of~$z_0$ under~$f$ is a zero of the power series
$$ h(z) \= \frac{f^{q p^n}(z) - z}{f^{q p^{n - 1}}(z) - z}, $$
and that the constant term of this power series is
$$ \frac{\lambda^{q p^n} - 1}{\lambda^{q p^{n - 1}} - 1}. $$
On the other hand, $\cO$ consists of~$q p^n$ points, and, since~$f$ maps~$\mK$ to itself isometrically, all the points in~$\cO$ have the same norm.
Applying Lemma~\ref{l:elementary division} inductively with~$\xi$ replaced by each element of~$\cO$, it follows that~$\prod_{z_0' \in \cO} (z - z_0')$ divides~$h(z)$ in~$\OK[[z]]$.
In particular, the constant term
\begin{equation}
\label{e:total product of periodic points}
\left( \frac{\lambda^{q p^n} - 1}{\lambda^{q p^{n - 1}} - 1} \right) / \prod_{z_0' \in \cO} (- z_0')
\end{equation}
of the power series $h(z) / \prod_{z_0' \in \cO} (z - z_0')$ is in~$\OK$.
We thus have
\begin{equation}
\label{e:powered lower bound}
|z_0|^{q p^n} = \prod_{z_0' \in \cO} |z_0'|
\ge
\left| \frac{\lambda^{q p^n} - 1}{\lambda^{q p^{n - 1}} - 1} \right|,
\end{equation}
and therefore~\eqref{e:periodic bound bis}.
Note that equality holds if and only if the constant term~\eqref{e:total product of periodic points} of the power series~$h(z) / \prod_{z_0' \in \cO} (z - z_0')$ has norm equal to~$1$.
Equivalently, equality holds if and only if~$\wideg \left(h(z) / \prod_{z_0' \in \cO} (z - z_0') \right) = 0$.
Using
$$ \wideg \left(h(z) / \prod_{z_0' \in \cO} (z - z_0') \right)
=
\wideg \left( \frac{f^{q p^n}(z) - z}{f^{q p^{n - 1}}(z) - z} \right) - q p^n, $$
we conclude that equality holds if and only if~$\wideg \left( \frac{f^{q p^n}(z) - z}{f^{q p^{n - 1}}(z) - z} \right) = q p^n$.
Finally, note that if this last equality holds, then~$\cO$ is the set of all zeros of~$\frac{f^{q p^n}(z) - z}{f^{q p^{n - 1}}(z) - z}$ in~$\mK$, so~$\cO$ is the only cycle of minimal period~$qp^n$ of~$f$.
This completes the proof of part~$2$.
\end{proof}

\section{Minimally ramified power series}
\label{s:minimally ramified}
Our main goal in this section is study condition~\eqref{e:minimal wideg increment} appearing in the optimality part of Lemma~\ref{l:periodic bound}.
To do this, for a given prime number~$p$ and a field~$k$ of characteristic~$p$, define for each power series~$g_0(\zeta) = \zeta + \cdots$ in~$k[[\zeta]]$ and each integer~$n \ge 0$, the order
$$ i_n(g_0)
\=
\ord \left( \frac{g_0^{p^n}(\zeta) - \zeta}{\zeta} \right). $$
Note that if~$K$, $q$, $\lambda$, and~$f$ are as in Lemma~\ref{l:periodic bound} and~$k = \tK$, then~$\tlambda$ is a root of unity of order~$q$ in~$k$, and for each integer~$n \ge 1$ such that
$$ \wideg \left( f^{qp^{n - 1}}(z) - z \right)
=
i_{n - 1} (\widetilde{f}^q) + 1 $$
is finite, we have
$$ \wideg \left( \frac{f^{qp^n}(\zeta) - \zeta}{f^{qp^{n - 1}}(\zeta) - \zeta} \right)
=
i_n(\widetilde{f}^q) - i_{n - 1}(\widetilde{f}^q). $$
Thus, \eqref{e:minimal wideg increment} naturally leads us to consider, for a root of unity~$\gamma$ of order~$q$ in~$k$ and a power series~$g(\zeta) = \gamma \zeta + \cdots$ in~$k[[\zeta]]$, the sequence~$\{ i_n(g^q) \}_{n = 0}^{+ \infty}$.
When~$p$ is odd, we show that for an integer~$n \ge 1$ such that~$i_{n - 1}(g^q)$ is finite, the equality
\begin{equation}
\label{e:minimally ramified increment}
i_n(g^q) - i_{n - 1}(g^q) = qp^n
\end{equation}
can only hold if~$g$ is ``minimally ramified'', in the sense that the sequence~$\left\{ i_n(g^q) \right\}_{n = 0}^{+ \infty}$ is the smallest possible, see Corollary~\ref{c:minimally ramified increment} in~\S\ref{ss:almost minimally ramified}, which also includes a characterization in the case~$p = 2$.
Thus, in the case~$p$ is odd and~$f(z) = \lambda z + \cdots$ is a polynomial in~$\OK[z]$ with non-linear reduction, the existence of an optimal cycle of period~$qp^n$ implies that~$\wtf$ is minimally ramified, see Corollary~\ref{c:optimality implies minimality}.

The structure of this section is as follows.
In~\S\ref{ss:higher order sen} we establish a ``higher order'' version of the main theorem of Sen in~\cite{Sen69}.
In~\S\ref{ss:minimally ramified} we combine this result with a result of Laubie and Sa{\"\i}ne in~\cite{LauSai98}, extending a previous result of Keating in~\cite{Kea92}, to characterize minimally ramified power series.
In~\S\ref{ss:almost minimally ramified} introduce the notion of ``almost minimally ramified'' power series, and we use it to handle the case~$p = 2$.

\subsection{A higher order version of Sen's theorem}
\label{ss:higher order sen}
The purpose of this section is to prove the following theorem.
\begin{theoalph}
\label{t:higher order sen}
Let~$p$ be a prime number, and~$k$ a field of characteristic~$p$.
Moreover, let~$\gamma$ be a root of unity in~$k$, $q \ge 1$ the order of~$\gamma$, and
$$ g(\zeta) = \gamma \zeta + a_2 \zeta^2 + \cdots $$ 
a power series in~$k[[\zeta]]$.
Then~$i_0(g^q)$ is divisible by~$q$ when finite.
Furthermore, for every integer~$n \ge 1$ such that~$i_n(g^q)$ is finite, $i_{n - 1}(g^q)$ is also finite and
$$ i_n(g^q)
\equiv
i_{n - 1}(g^q) \mod q p^n. $$
In particular, for every~$n \ge 0$ such that~$i_n(g^q)$ is finite, $i_n(g^q)$ is divisible by~$q$.
\end{theoalph}
When restricted to~$q = 1$, the theorem above is~\cite[Theorem~$1$]{Sen69}.
Sen's original proof in~\cite{Sen69} is based on a careful analysis of the orders of cocycles of power series in~$k[[\zeta]]$.
Lubin gave a conceptual proof of this result in~\cite{Lub95}, that is even shorter than Sen's original proof; Lubin interprets~$i_n(g) - i_{n - 1}(g)$ as the number of periodic points of minimal period~$p^n$ of a certain ``lift'' of~$g$.
See also~\cite[Theorem~$3.1$]{Li96c} for a variant of Lubin's proof.

To prove Theorem~\ref{t:higher order sen}, we follow Lubin's strategy.
The main difficulty is to find, for a given~$n$, a lift~$g$ such that the zeros of~$g^{qp^n}(z) - z$ are simple.
Lubin achieved this through an inductive perturbative procedure.
We use the fact that a generic polynomial has no parabolic periodic point.

\begin{lemm}
\label{l:generic independence zero characteristic}
Let~$K$ be a field of characteristic zero, $d \ge 2$ an integer, and~$a_1$, \ldots, $a_d$ in~$K$ algebraically independent over the prime field of~$K$.
Then the polynomial
$$ a_1 z + \cdots + a_d z^d $$
in~$K[z]$ has no parabolic periodic point.
\end{lemm}
\begin{proof}
Denote by~$\Q$ the prime field of~$K$, and by~$| \cdot |$ the usual absolute value in~$\C$.

Suppose by contradiction there is an integer~$n \ge 1$ and a periodic point~$z_0$ of period~$n$ of the polynomial~$P(z) \= a_1 z + \cdots + a_d z^d$ in~$K[z]$, such that~$(P^n)'(z_0)$ is a root of unity.
Let~$\sigma : \Q[z_0, a_1, \ldots, a_d] \to \C$ be a ring homomorphism such that~$\sigma(a_d) = 1$, and such that for each~$j$ in~$\{1, \ldots, d - 1 \}$ we have~$\sigma(a_j) = 0$.
Then~$\sigma(P)(z) = z^d$, $\sigma(z_0)$ is a periodic point of period~$n$ of~$\sigma(P)$, and~$(\sigma(P)^n)'(\sigma(z_0)) = \sigma((P^n)'(z_0))$ is a root of unity.
This implies that~$\sigma(z_0) \neq 0$, and therefore that~$|\sigma(z_0)| = 1$.
Thus,
$$ \left| (\sigma(P)^n)'(\sigma(z_0)) \right|
=
\left| d^n \sigma(z_0)^{d^n - 1} \right|
=
d^n. $$
This contradicts our hypothesis that~$\sigma((P^n)'(z_0))$ is a root of unity, and proves the lemma.
\end{proof}
\begin{proof}[Proof of Theorem~\ref{t:higher order sen}]
Replacing~$k$ by one of its algebraic closures if necessary, assume~$k$ is algebraically closed.
Then~$k$ is perfect and therefore there is an algebraically closed field~$K$ of characteristic zero that is complete with respect to a non-trivial ultrametric norm and whose residue field~$\tK$ is isomorphic to~$k$, see for example~\cite[II, \emph{Th{\'e}or{\`e}me}~$3$]{Ser68a}.
Identify~$k$ with~$\tK$.
Then~$K$ is uncountable and therefore we can choose for each~$j$
in~$\{ 1, \ldots, i_n(g^q) + 1 \}$ an element~$a_j$ of~$K$, such that the~$a_1$, \ldots, $a_{i_n(g^q) + 1}$ are algebraically independent over the prime field of~$K$ and such that the reduction~$\tP$ of the polynomial
$$ P(z) = a_1 z + \cdots + a_{i_n(g^q) + 1} z^{i_n(g^q) + 1} $$
in~$K[z]$, satisfies $\tP(\zeta) \equiv g(\zeta) \mod \left\langle \zeta^{i_n(g^q) + 2} \right\rangle$ in~$k[\zeta]$.
Then
$$ \wideg \left( P^{qp^n}(z) - z \right) = i_n(g^q) + 1 $$
and by Lemma~\ref{l:generic independence zero characteristic} the polynomial~$P$ has no parabolic periodic points.

Suppose~$n = 0$.
From~$a_1^q \neq 1$, it follows that~$P^q(z) - z$ has precisely~$i_0(g^q)$ zeros in~$\mK \setminus \{ 0 \}$, counted with multiplicity.
Note that if~$P^q(z) - z$ had a double zero of~$z_0$ in~$\mK \setminus \{ 0 \}$, then~$z_0$ would also be a zero of~$(P^q)'(z) - 1$, and therefore~$z_0$ would be a parabolic periodic point of~$P$.
We conclude that all zeros of~$P^q(z) - z$ in~$\mK \setminus \{ 0 \}$ are simple, and therefore that~$P^q(z) - z$ has precisely~$i_0(g^q)$ zeros in~$\mK \setminus \{ 0 \}$.
By part~$2$ of Lemma~\ref{l:admissible periods}, every zero of~$P^q(z) - z$ in~$\mK \setminus \{ 0 \}$ is a periodic point of minimal period~$q$ of~$P$.
Combined with~$P(\mK) = \mK$, it follows that the set~$Z_0$ of zeros of~$P^q(z) - z$ in~$\mK \setminus \{ 0 \}$ is a union of periodic orbits of minimal period~$q$.
We conclude that~$\# Z_0 = i_0(g^q)$ is divisible by~$q$.
This completes the proof of the theorem in the case~$n = 0$.

Suppose~$n \ge 1$.
Our assumption that~$i_n(g^q)$ is finite, together with the straight forward inequality~$i_{n - 1}(g^q) \le i_n(g^q)$, implies that~$i_{n - 1}(g^q)$ is also finite.
So, by our choice of~$P$, we have
$$ \wideg \left( P^{qp^{n - 1}}(z) - z \right) = i_{n - 1}(g^q) + 1, $$
and therefore $h(z) \= \frac{P^{qp^n}(z) - z}{P^{qp^{n - 1}}(z) - z}$ has precisely $i_n(g^q) - i_{n - 1}(g^q)$ zeros in~$\mK$, counted with multiplicity.
As in the previous case, if~$P^{qp^n}(z) - z$ had a double zero~$z_0$, then~$z_0$ would also be a zero of~$(P^{qp^n})'(z) - 1$, and therefore~$z_0$ would be a parabolic periodic point of~$P$.
We conclude that all zeros of~$P^{qp^n}(z) - z$, and hence of~$h$, are simple.
In particular, $h$ has precisely~$i_n(g^q) - i_{n - 1}(g^q)$ zeros in~$\mK$.
It also follows that a zero of~$h$ cannot be a zero of~$P^{qp^{n - 1}}(z) - z$.
In view of part~$2$ of Lemma~\ref{l:admissible periods}, this implies that the zeros of~$h$ are precisely the periodic points of~$P$ of minimal period~$qp^n$.
Since~$P(\mK) = \mK$, it follows that the set~$Z_n$ of zeros of~$h$ in~$\mK$ is a union of periodic orbits of minimal period~$qp^n$.
We conclude that~$\# Z_n = i_n(g^q) - i_{n - 1}(g^q)$ is divisible by~$qp^n$.
This completes the proof of the theorem.
\end{proof}

\subsection{Minimally ramified power series}
\label{ss:minimally ramified}
In this section we introduce the notion of ``minimally ramified'' power series, that is motivated by the following proposition.
\begin{prop}
  \label{p:minimally ramified}
Let~$p$ be a prime number, $k$~a field of characteristic~$p$, and~$\gamma$ a root of unity in~$k$.
If we denote by~$q$ the order of~$\gamma$, then for every power series~$g(\zeta) = \gamma \zeta + \cdots$ in~$k[[\zeta]]$ and every integer~$n \ge 0$, we have
\begin{equation}
  \label{e:minimally ramified at n}
i_n(g^q)
\ge
q \frac{p^{n + 1} - 1}{p - 1}.  
\end{equation}
If~$p$ is odd (resp.~$p = 2$) and equality holds for some~$n \ge 1$ (resp.~$n \ge 2$), then equality holds for every~$n \ge 0$.
\end{prop}
\begin{rema}
  \label{r:minimally ramified}
In contrast with the case where~$p$ is odd, when~$p = 2$ equality in~\eqref{e:minimally ramified at n} for~$n = 1$ does not necessarily imply that we have equality in~\eqref{e:minimally ramified at n} for every~$n \ge 0$.
In fact, suppose~$p = 2$ and put~$g(\zeta) \= \gamma \zeta (1 + \zeta^q)$ if~$q \equiv 1 \mod 4$, and~$g(\zeta) \= \gamma \zeta (1 + \zeta^q + \zeta^{2q})$ if~$q \equiv - 1 \mod 4$.
Then a direct computation shows that~$i_1(g^q) = 3 q$ and~$i_2(g^q) > 7q$.
\end{rema}

The proof of Proposition~\ref{p:minimally ramified} is at the end of this section.

Motivated by Proposition~\ref{p:minimally ramified}, and following the terminology introduced by Laubie, Movahhedi, and Salinier in~\cite{LauMovSal02} in the case~$q = 1$, we make the following definition.

\begin{defi}
\label{d:minimally ramified}
Let~$p$ be a prime number, $k$ a field of characteristic~$p$, $\gamma$ a root of unity in~$k$, and~$q$ the order of~$\gamma$.
Then a power series~$g(\zeta) = \gamma \zeta + \cdots$ in $k[[\zeta]]$ is \emph{minimally ramified}, if for every integer~$n \ge 0$ we have
$$ i_n(g^q) = q \frac{p^{n + 1} - 1}{p - 1}. $$
\end{defi}

To prove Proposition~\ref{p:minimally ramified}, we use several times the following consequence of~\cite[Corollary~$1$]{LauSai98}, see also~\cite[Theorem~$7$]{Kea92} for the case~$q = 1$.
\begin{lemm}
\label{l:laubie saine}
Let~$p$ be a prime number, $k$ a field of characteristic~$p$, $\gamma$ a root of unity, and~$g(\zeta) = \gamma \zeta + \cdots $ a power series in~$k[[\zeta]]$.
If~$p$ is odd (resp.~$p = 2$), then~$g$ is minimally ramified if and only if~\eqref{e:minimally ramified at n} holds with equality for~$n = 0$ and~$n = 1$ (resp.~$n = 0$, $n = 1$, and~$n = 2$).
\end{lemm}

We also use the following lemma several times, see also~\cite[\emph{Exercice}~$3$, \S IV]{Ser68a} or~\cite[Lemma~$3$]{Kea92} for the case~$q = 1$.

\begin{lemm}
\label{l:keating}
Let~$p$ be a prime number, $k$ a field of characteristic~$p$, $\gamma$ a root of unity in~$k$, and~$q$ the order of~$\gamma$.
Then for each power series
$$ g(\zeta) = \gamma \zeta + \cdots $$
in~$k[[\zeta]]$ and every integer~$n \ge 0$ such that~$i_n(g^q)$ is finite, the following properties hold:
\begin{enumerate}
\item[1.]
If~$i_n(g^q)$ is not divisible by~$p$, then~$i_{n + 1}(g^q) \ge p i_n(g^q) + q$;
\item[2.]
If~$i_n(g^q)$ is divisible by~$p$, then~$i_{n + 1}(g^q) = p i_n(g^q)$.
\end{enumerate}
\end{lemm}
\begin{proof}
For each integer~$m \ge 1$ define the power series~$\Delta_m(\zeta)$ inductively by~$\Delta_1(\zeta) \= g^{qp^n}(\zeta) - \zeta$, and for~$m \ge 2$ by
$$ \Delta_m(\zeta)
\=
\Delta_{m - 1}(g^{qp^n}(\zeta)) - \Delta_{m - 1}(\zeta). $$
An induction argument shows that
$$ \Delta_m (\zeta)
=
\sum_{j = 0}^m \binom{m}{j} (-1)^{m - j} g^{qp^nj}(\zeta). $$
Taking~$m = p$, we obtain~$\Delta_p(\zeta) = g^{qp^{n + 1}}(\zeta) -
\zeta$.
Noting that~$i \= \ord(\Delta_1)$ is equal to~$i_n(g^q) + 1$, put
$$ \Delta_1(\zeta) = \sum_{j = i}^{+ \infty} a_j \zeta^j, $$
so that~$a_i \neq 0$.

Given an integer~$m \ge 1$, put
$$ o \= \ord(\Delta_m)
\text{ and }
\Delta_m(\zeta) = \sum_{j = o}^{+ \infty} b_j \zeta^j, $$
so that~$b_o \neq 0$.
Then we have
\begin{equation*}
  \begin{split}
    \Delta_{m + 1}(\zeta)
& =
\sum_{j = o}^{+ \infty} b_j \left( \left( \zeta + \Delta_1(\zeta) \right)^j - \zeta^j \right)
\\ & =
\sum_{j = o}^{+ \infty} b_j \left( \zeta^j \left(1 + a_i \zeta^{i - 1} +
    \cdots \right)^j - \zeta^j \right)
\\ & \equiv
b_o a_i o \zeta^{o + i - 1} \mod \left\langle \zeta^{o + i} \right\rangle
  \end{split}
\end{equation*}
in~$k[[\zeta]]$.
It follows that
\begin{equation}
  \label{e:inductive order}
\ord (\Delta_{m + 1}) \ge \ord (\Delta_m) + i_n(g^q),  
\end{equation}
with equality if and only if~$\ord (\Delta_m)$ is not divisible by~$p$.
Since~$\ord(\Delta_1) = i_n(g^q) + 1$, when~$i_n(g^q)$ is divisible
by~$p$ for every integer~$m \ge 1$ we have~$\ord(\Delta_m) = m
i_n(g^q) + 1$.
Taking~$m = p$ and using~$\ord(\Delta_p) = i_{n + 1}(g^q) + 1$, we conclude that~$i_{n + 1}(g^q) = p i_n(g^q)$.
This proves part~$2$.
To prove part~$1$, suppose that~$i_n(g^q)$ is not divisible by~$p$ and that~$i_{n + 1}(g^q)$ is finite.
Let~$\ell$ be the integer in~$\{1, \ldots, p - 1 \}$ such that $\ell
\cdot i_n(g^q) \equiv - 1 \mod \left\langle p \right\rangle$.
Applying~\eqref{e:inductive order} inductively, we obtain that for
every~$m$ in~$\{1, \ldots, \ell \}$ we have~$\ord(\Delta_m) = m
i_n(g^q) + 1$.
Since~$\ell \cdot i_n(g^q) + 1$ is divisible by~$p$,
by~\eqref{e:inductive order} with~$m = \ell$ we
have~$\ord(\Delta_{\ell + 1}) \ge (\ell + 1) i_n(g^q) + 2$.
In the case~$\ell = p - 1$ we obtain~$\ord(\Delta_p) \ge p
i_n(g^q) + 2$.
If~$\ell \neq p - 1$, then using~\eqref{e:inductive order} inductively
we also obtain~$\ord(\Delta_p) \ge p i_n(g^q) + 2$.
So this last inequality holds in all cases.
Using~$\ord(\Delta_p) = i_{n + 1}(g^q) + 1$ and the fact that~$i_{n +
  1}(g^q)$ and~$i_n(g^q)$ are both divisible by~$q$ by
Theorem~\ref{t:higher order sen}, we conclude that~$i_{n + 1}(g^q) \ge
p i_n(g^q) + q$.
This proves part~$1$ and completes the proof of the lemma.
\end{proof}
\begin{proof}[Proof of Proposition~\ref{p:minimally ramified}]
To prove that we have~\eqref{e:minimally ramified at n} for every~$n \ge 0$, we proceed by induction.
The case~$n = 0$ follows from the fact that~$i_0(g^q)$ is divisible by~$q$ when finite, see Theorem~\ref{t:higher order sen}.
Let~$n \ge 0$ be an integer for which~\eqref{e:minimally ramified at n} holds, and suppose that~$i_{n + 1}(g^q)$, and hence~$i_n(g^q)$, is finite.
Using~$i_{n + 1}(g^q) \ge i_n(g^q) + 1$ (\emph{cf.}
Lemma~\ref{l:keating}) and that~$i_{n + 1}(g^q) - i_n(g^q)$ is
divisible by~$qp^{n + 1}$ (Theorem~\ref{t:higher order sen}), we have
$$ i_{n + 1}(g^q)
\ge
i_n(g^q) + q p^{n + 1}
\ge
q \frac{p^{n + 1} - 1}{p - 1} + q p^{n + 1}
=
q \frac{p^{n + 2} - 1}{p - 1}. $$
This completes the proof of the induction step, and that~\eqref{e:minimally ramified at n} holds for every~$n \ge 0$.

To prove the last part of the proposition, suppose~$p$ is odd
(resp.~$p = 2$) and that for some~$n \ge 1$ (resp.~$n \ge 2$) we have~$i_n(g^q) = q \frac{p^{n + 1} - 1}{p - 1}$.
We prove by induction that for every~$\ell$ in~$\{0, \ldots, n \}$ we have~$i_{n - \ell}(g^q) = q \frac{p^{n - \ell + 1} - 1}{p - 1}$.
When~$\ell = 0$ this holds by hypothesis.
Suppose this holds for some~$\ell$ in~$\{0, \ldots, n - 1 \}$.
In particular~$i_{n - \ell}(g^q)$ is not divisible by~$p$, and by
part~$2$ of Lemma~\ref{l:keating} the number~$i_{n - \ell - 1}(g^q)$ is not divisible by~$p$ either.
Thus, by part~$1$ of the same lemma we have
$$ p i_{n - \ell - 1}(g^q) + q \le i_{n - \ell}(g^q) = q \frac{p^{n - \ell + 1} - 1}{p - 1}, $$
and therefore~$i_{n - \ell - 1}(g^q) \le q \frac{p^{n - \ell} - 1}{p - 1}$.
Since we already proved the reverse inequality, we obtain~$i_{n - \ell - 1}(g^q) = q \frac{p^{n - \ell} - 1}{p - 1}$.
This completes the proof of the induction step, and of the fact that for every~$\ell$ in~$\{0, \ldots, n \}$ we have~$i_{n - \ell}(g^q) = q \frac{p^{n - \ell + 1} - 1}{p - 1}$.
Combined with Lemma~\ref{l:laubie saine}, this implies the last part of the proposition.
\end{proof}

\subsection{Almost minimally ramified power series}
\label{ss:almost minimally ramified}
For a ground field of characteristic~$2$, in this section we study those power series that are ``almost minimally ramified'' (Proposition~\ref{p:almost minimally ramified} and Definition~\ref{d:almost minimally ramified}).
We use this and the results in~\S\ref{ss:minimally ramified}, to characterize in arbitrary characteristic the occurrence of~\eqref{e:minimally ramified increment} in terms of (almost) minimally ramified power series (Corollary~\ref{c:minimally ramified increment}).
In turn, this allows us to show that, in some cases, the existence of an optimal cycle implies that the reduction of the map is (almost) minimally ramified (Corollary~\ref{c:optimality implies minimality}).
\begin{prop}
  \label{p:almost minimally ramified}
Let~$k$ be a field of characteristic~$2$, $\gamma$ a root of unity
in~$k$, $q$ the order of~$\gamma$, and~$g(\zeta) = \gamma \zeta +
\cdots$ a power series in~$k[[\zeta]]$.
Then the following properties hold:
\begin{enumerate}
\item[1.]
If~$g$ is not minimally ramified, then for every integer~$n \ge 2$ we have
\begin{equation}
\label{e:almost minimally ramified at n}
i_n(g^q)
\ge
2^{n + 1} q;
\end{equation}
\item[2.]
If equality holds in~\eqref{e:almost minimally ramified at n} for~$n =
0$ or for some~$n \ge 2$, then it holds for every~$n \ge 0$.
\end{enumerate}
\end{prop}

The proof of this proposition is at the end of this section.

\begin{defi}
\label{d:almost minimally ramified}
Let~$k$ be a field of characteristic~$2$, $\gamma$ a root of unity in~$k$, and~$q$ the order of~$\gamma$.
Then a power series~$g(\zeta) = \gamma \zeta + \cdots$ in $k[[\zeta]]$ is \emph{almost minimally ramified}, if for every integer~$n \ge 0$ we have
$$ i_n(g^q) = 2^{n + 1} q. $$
\end{defi}

The following is a direct consequence of Proposition~\ref{p:almost minimally ramified}.
\begin{coro}
\label{c:(almost) minimally ramified}
Let~$k$ be a field of characteristic~$2$, $\gamma$ a root of unity in~$k$, and~$q$ the order of~$\gamma$.
If~$g(\zeta) = \gamma \zeta + \cdots$ is a power series in~$k[[\zeta]]$ such that for some integer~$n \ge 2$ we have~$i_n(g^q) \le 2^{n + 1} q$, then~$g$ is either minimally ramified or almost minimally ramified.
\end{coro}

The following corollary is a consequence of Propositions~\ref{p:minimally ramified} and~\ref{p:almost minimally ramified}.
\begin{coro}
\label{c:minimally ramified increment}
Let~$p$ be a prime number, and $k$~a field of characteristic~$p$.
Then for every root of unity~$\gamma$ in~$k$ of order~$q$, and every power series~$g(\zeta) = \gamma \zeta + \cdots$ in~$k[[\zeta]]$, the following properties hold.
\begin{enumerate}
\item[1.]
Suppose~$p$ is odd and that for some integer~$n \ge 1$ such that~$i_{n - 1}(g^q)$ is finite, we have~$i_n(g^q) - i_{n - 1}(g^q) = qp^n$.
Then~$g$ is minimally ramified.
\item[2.]
Suppose~$p = 2$ and that for some~$n \ge 2$ such that~$i_{n - 1}(g^q)$ is finite, we have~$i_n(g^q) - i_{n - 1}(g^q) = 2^n q$.
Then~$g$ is either minimally ramified, or almost minimally ramified.
\end{enumerate}
\end{coro}
\begin{proof}
To prove part~$1$, suppose~$p$ is odd and let~$n \ge 1$ be such that~$i_{n - 1}(g^q)$ is finite and~$i_n(g^q) = i_{n - 1}(g^q) + qp^n$.
Suppose by contradiction that~$i_{n - 1}(g^q)$ is divisible by~$p$.
Then by part~$2$ of Lemma~\ref{l:keating} we have~$i_n(g^q) = p i_{n - 1}(g^q)$, so
$$ qp^n = i_n(g^q) - i_{n - 1}(g^q) = (p - 1) i_{n - 1}(g^q). $$
Since~$i_{n - 1}(g^q)$ is divisible by~$q$ (Theorem~\ref{t:higher order sen}), this implies that~$p - 1$ divides~$p^n$.
However, this is not possible because~$p - 1$ is even and~$p^n$ is odd.
We conclude that~$i_{n - 1}(g^q)$ is not divisible by~$p$.
Then part~$1$ of Lemma~\ref{l:keating} implies that~$i_n(g^q) \ge p i_{n - 1}(g^q) + q$, so
$$ i_n(g^q)
=
i_{n - 1}(g^q) + q p^n
\le
(i_n(g^q) - q) / p + q p^n, $$
and~$i_n(g^q) \le q \frac{p^{n + 1} - 1}{p - 1}$.
Then Proposition~\ref{p:minimally ramified} implies that~$g$ is minimally ramified.
This proves part~$1$.

To prove part~$2$, suppose~$p = 2$ and let~$n \ge 2$ be such that~$i_{n - 1}(g^q)$ is finite and~$i_n(g^q) = i_{n - 1}(g^q) + 2^n q$.
By Lemma~\ref{l:keating} we have
$$ i_n(g^q)
=
i_{n - 1}(g^q) + 2^n q
\le
i_n(g^q) / 2 + 2^n q. $$
We thus have~$i_n(g^q) \le 2^{n + 1} q$, and by Corollary~\ref{c:(almost) minimally ramified} the power series~$g$ is either minimally ramified or almost minimally ramified.
\end{proof}

\begin{coro}
\label{c:optimality implies minimality}
Let~$p$, $K$, $\lambda$, and~$q$ be as in Lemma~\ref{l:periodic bound} and let~$n \ge 1$ be an integer and~$P(z) = \lambda z + \cdots $ in~$\OK[z]$ a polynomial having an optimal cycle of period~$qp^n$.
Then the following properties hold:
\begin{enumerate}
\item[1.]
If~$p$ is odd, then~$\tP$ is minimally ramified;
\item[2.]
If~$p = 2$ and~$n \ge 2$, then~$\tP$ is either minimally ramified, or almost minimally ramified.
\end{enumerate}
\end{coro}
\begin{proof}
If the reduction of~$P$ is non-linear, then for every integer~$n$ the
order~$i_n(\tP^q)$ is finite, and therefore the assertions are direct consequences of Lemma~\ref{l:periodic bound} and Corollary~\ref{c:minimally ramified increment}.
Thus, to complete the proof of the corollary we just need to show that the reduction of~$P$ is non-linear.
Suppose by contradiction this is not the case.
Extending~$K$ if necessary, we assume it is algebraically closed.
Then there is~$\mu$ in~$\mK \setminus \{ 0 \}$ such that the polynomial~$Q(w) \= \mu^{-1} P(\mu w)$ is in~$\OK[w]$.
Note that the map~$M_\mu(w) = \mu w$ maps the periodic points of~$Q$ to those of~$P$ preserving minimal periods.
Thus, applying Lemma~\ref{l:periodic bound} to~$Q$, we conclude that~$P$ cannot have an optimal cycle.
This contradiction proves that the reduction of~$P$ is non-linear and completes the proof of the corollary.
\end{proof}
\begin{proof}[Proof of Proposition~\ref{p:almost minimally ramified}]
To prove~\eqref{e:almost minimally ramified at n} with~$n = 2$, note first that by Proposition~\ref{p:minimally ramified} with~$n = 1$ we have~$i_1(g^q) \ge 3 q$.
Suppose~$i_1(g^q) = 3 q$, and note that by Theorem~\ref{t:higher order sen} we have~$i_0(g^q) = q$, and either~$i_2(g^q) = 7 q$ or~$i_2(g^q) \ge 11 q$.
But we cannot have~$i_2(g^q) = 7 q$, for otherwise Lemma~\ref{l:laubie saine} would imply that~$g$ is minimally ramified.
Thus, $i_2(g^q) \ge 11 q$.
This proves~\eqref{e:almost minimally ramified at n} with~$n = 2$ when~$i_1(g^q) = 3 q$.
If~$i_1(g^q) > 3 q$, then by Theorem~\ref{t:higher order sen} we
have~$i_1(g^q) \ge 4 q$, and by Lemma~\ref{l:keating} we have~$i_2(g^q) \ge 2 i_1(g^q) \ge 8 q$.
This proves that in all the cases we have~\eqref{e:almost minimally ramified at n} with~$n = 2$.
For~$n \ge 3$ inequality~\eqref{e:almost minimally ramified at n} is
then obtained by applying Lemma~\ref{l:keating} inductively.
This completes the proof of part~$1$.

To prove part~$2$, suppose~$i_0(g^q) = 2q$.
Applying Lemma~\ref{l:keating} repeatedly we conclude that for every~$n \ge 1$
we have~$i_n(g^q) = 2^{n + 1}q$.
Suppose now that for some~$n \ge 2$ we have~$i_n(g^q) = 2^{n + 1} q$.
If~$n \ge 3$, then applying Lemma~\ref{l:keating} repeatedly we obtain
\begin{multline}
\label{e:(almost) minimally ramified}
i_2(g^q) \le i_n(g^q) / 2^{n - 2} = 8q,
i_1(g^q) \le i_2(g^q) /2 \le 4 q,
\\ \text{ and }
i_0(g^q) \le i_1(g^q) / 2 \le 2q.
\end{multline}
Together with Theorem~\ref{t:higher order sen}, this implies either~$i_0(g^q) = q$ or~$i_0(g^q) = 2q$.
In the latter case we obtain the desired conclusion by applying part~$2$ of Lemma~\ref{l:keating} repeatedly.
It remains to consider the case~$i_0(g^q) = q$.
Since~$i_1(g^q) \le 4q$ and~$i_2(g^q) \le 8q$, by Theorem~\ref{t:higher order sen} we must have~$i_1(g^q) = 3q$ and~$i_2(g^q) = 7 q$.
However, by Lemma~\ref{l:laubie saine} this implies that~$g$ is minimally ramified.
We thus obtain a contradiction that completes the proof of part~$2$
and of the proposition.
\end{proof}

\section{Characterizing minimally ramified power series}
\label{s:characterization of minimally ramified}
In this section we give a characterization of minimally ramified power
series (Theorem~\ref{t:characterization of minimally ramified}).
This characterization is best expressed in terms of the iterative
residue, which is a conjugacy invariant introduced by {\'E}calle in
the complex setting, see~\cite{Eca75}.
We define this invariant for a restricted class of power series that is sufficient for our purposes.

Let~$p$ be a prime number and $k$ a field of characteristic~$p$.
Denote by~$\sK_k$ the set of power series~$g(\zeta)$ in~$k[[\zeta]]$ satisfying~$g(0) = 0$ and~$g'(0) \neq 0$.
It is a group under composition.
We say that~$2$ power series~$g(\zeta)$ and~$\whg(\zeta)$ in~$\sK_k$
are \emph{conjugated} if there is a power series~$h(\zeta)$ in~$\sK_k$ such
that~$\whg(\zeta) = h \circ g \circ h^{-1}(\zeta)$.
Note that in this case we have~$\whg'(0) = g'(0)$.
Moreover, if~$\gamma \= g'(0)$ is a root of unity and we denote by~$q$
its order, then for every integer~$n \ge 0$ we have~$i_n(g^q) = i_n
(\whg^q)$.\footnote{In fact, $i_n(g^q) + 1$ is equal to the
  multiplicity of~$\zeta = 0$ as a fixed point of~$g^{qp^n}(\zeta)$,
  and this is clearly invariant under conjugacy.}
In particular, minimal ramification is invariant under conjugacy.

Let~$\gamma$ be a root of unity in~$k$, let~$q$ be its order, and
let~$g(\zeta)$ be a power series in~$k[[\zeta]]$ satisfying~$g'(0) = \gamma$.
In the case~$\gamma = 1$, so that~$q = 1$, put
$$ g(\zeta)
=
\zeta \left( 1 + a_1 \zeta + a_2 \zeta^2 + \cdots \right), $$
and assume~$a_1 \neq 0$.
Then the \emph{iterative residue $\resit(g)$\footnote{We use {\'E}calle's notation ``$\resit$'', which is an abbreviation of the French term ``\emph{r{\'e}sidue it{\'e}ratif}'' corresponding to ``iterative residue''.} of~$g$} is
\begin{equation*}
\resit(g) \= 1 - \frac{a_2}{a_1^2}.
\end{equation*}
Note that the condition~$a_1 \neq 0$ is equivalent to~$i_0(g) = 1$.
To define the iterative residue in the case~$\gamma \neq 1$, so
that~$q \ge 2$, we use the fact that~$g(\zeta)$ is conjugated to a power series of the form
$$ \whg(\zeta)
=
\gamma \zeta \left( 1 + \sum_{j = 1}^{+ \infty} a_j \zeta^{j q} \right), $$
see Proposition~\ref{p:basic normal form}.
In general, the power series~$\whg$ is not uniquely determined by~$g$.
To define the ``iterative residue'' of~$g$ we restrict to the case when~$a_1
\neq 0$.
This last condition is equivalent to~$i_0(g^q) = q$ (Proposition~\ref{p:basic normal form}), so it only
depends on~$g$.
When this condition is satisfied, the quotient~$a_2/a_1^2$ only
depends on~$g$ (Proposition~\ref{p:basic normal form}) and we
define the \emph{iterative residue of~$g$} by
\begin{equation*}
\resit(g) \= \frac{q + 1}{2} - \frac{a_2}{a_1^2}.
\end{equation*}
Note that in the case~$p = 2$ the number~$q$ is odd, so the
quotient~$\frac{q + 1}{2}$ is an integer and it thus represents an element
of~$k$.
Note also that in this case we have~$q = 1$ in~$k$.
\begin{theoalph}
  \label{t:characterization of minimally ramified}
Let~$p$ be a prime number, $k$ a field of characteristic~$p$, and~$\gamma$
a root of unity in~$k$.
Moreover, let~$q$ be the order of~$\gamma$ and let~$g(\zeta)$ be a
power series in~$k[[\zeta]]$ of the form
$$ g(\zeta) = \gamma \zeta + \cdots . $$
If~$p$ is odd (resp.~$p = 2$), then~$g$ is minimally ramified if and
only if
$$ i_0(g^q) = q
\text{ and }
\resit(g) \neq 0 $$
$$ \left( \text{resp. } i_0(g^q) = q,
\resit(g) \neq 0,
\text{ and }
\resit(g) \neq 1 \right). $$
\end{theoalph}

When~$p$ is odd and~$q = 1$, Theorem~\ref{t:characterization of
  minimally ramified} is~\cite[\emph{Exemple}~$3.19$]{Rivthese}, phrased in terms of the iterative residue.

A direct consequence of Theorem~\ref{t:characterization of
  minimally ramified} is that for every integer~$q$ and every root of
unity~$\gamma$ of order~$q$ in a field of positive characteristic~$k$, there is a minimally ramified polynomial~$g(\zeta) = \gamma \zeta + \cdots$ in~$k[\zeta]$ of degree~$q + 1$ or~$2q + 1$.
This is exploited in~\S\ref{s:optimality}.

The proof of Theorem~\ref{t:characterization of minimally ramified} is
given in~\S\ref{ss:proof of characterization of minimally ramified},
after showing in~\S\ref{ss:conjugacy classes} the results needed to define the iterative residue.

\subsection{Conjugacy classes}
\label{ss:conjugacy classes}
The purpose of this section is to prove the following proposition that
was used above to define the iterative residue.
\begin{prop}
\label{p:basic normal form}
Let~$p$ be a prime number, $k$ a field of characteristic~$p$, $\gamma$
a root of unity in~$k$, and~$q$ the order of~$\gamma$.
Then every power series~$g(\zeta)$ in~$k[[\zeta]]$ satisfying~$g(\zeta) = \gamma \zeta + \cdots$ is conjugated to a power series of the form
$$ \whg(\zeta)
=
\gamma \zeta \left( 1 + \sum_{j = 1}^{+ \infty} a_j \zeta^{jq}
\right). $$
Moreover, we have~$a_1 \neq 0$ if and only if~$i_0(g^q) = q$, and in
this case the quotient~$a_2 / a_1^2$ depends only on~$g$.
\end{prop}

The proof of this proposition depends on a couple of lemmas.

The following lemma is stated in a stronger form than what is needed
for the proof of Proposition~\ref{p:basic normal form}; it is used in
the proofs of Propositions~\ref{p:odd optimality} and~\ref{p:even
  periodic optimality}.
\begin{lemm}
\label{l:q iterate}
Let~$p$ be a prime number, $k$ a field of characteristic~$p$, and~$q \ge 2$ an integer that is not divisible by~$p$.
Given~$\gamma$ in~$k^*$ satisfying~$\gamma^q = 1$, and~$a_1$ and~$a_2$ in~$k$, let~$g(\zeta)$ be a power series in~$k[[\zeta]]$ satisfying
$$ g(\zeta)
\equiv
\gamma \zeta \left(1 + a_1 \zeta^q + a_2 \zeta^{2q} \right)
\mod \left\langle \zeta^{2q + 2} \right\rangle $$
in~$k[[\zeta]]$.
Then  for every integer~$\ell \ge 1$, we have
\begin{equation}
  \label{e:iterated q series}
g^\ell(\zeta)
\equiv
\gamma^{\ell} \zeta \left( 1 + \ell a_1 \zeta^q + \left( \ell a_2 + (q + 1) \frac{\ell(\ell - 1)}{2} a_1^2 \right) \zeta^{2q} \right)
\mod \left\langle \zeta^{2q + 2} \right\rangle.
\end{equation}
\end{lemm}
\begin{proof}
We proceed by induction.
When~$\ell = 1$ the congruence~\eqref{e:iterated q series} holds by definition of~$g$.
Let~$\ell \ge 1$ be an integer for which~\eqref{e:iterated q series} holds.
Then, using~$q \ge 2$, we have
\begin{multline*}
  \begin{aligned}
g^{\ell + 1}(\zeta)
& \equiv
\gamma^\ell \left[ \gamma \zeta \left( 1 + a_1 \zeta^q + a_2 \zeta^{2q} \right) \right]
\\ & 
\cdot \left( 1 + \ell a_1 \zeta^q \left( 1 + a_1 \zeta^q \right)^q
+ \left( \ell a_2 + (q + 1) \frac{\ell(\ell - 1)}{2} a_1^2 \right) \zeta^{2q} \right)
\mod \left\langle \zeta^{2q + 2} \right\rangle
\\ & \equiv
\gamma^{\ell + 1} \zeta \left( 1 + a_1 \zeta^q + a_2 \zeta^{2q} \right)
\\ & 
\cdot \left( 1 + \ell a_1 \zeta^q + \left( \ell a_2 + \left( q \ell + (q +
    1) \frac{\ell(\ell - 1)}{2} \right) a_1^2 \right) \zeta^{2q} \right)
\mod \left\langle \zeta^{2q + 2} \right\rangle
\\ & \equiv
\gamma^{\ell + 1} \zeta \left( 1 + (\ell + 1) a_1 \zeta^q + \left( (\ell + 1) a_2 + (q + 1) \frac{\ell(\ell + 1)}{2} a_1^2 \right) \zeta^{2q} \right)
  \end{aligned}
\\
\mod \left\langle \zeta^{2q + 2} \right\rangle.
\end{multline*}
This proves the induction step, and completes the proof of the lemma.
\end{proof}

\begin{lemm}
\label{l:invariance of resit}
Let~$p$ be a prime number, $k$ a field of characteristic~$p$, $\gamma$
a root of unity in~$k$, and let~$q$ be the order of~$\gamma$. 
Given~$A$ and $\whA$ in~$k^*$ and~$B$ and~$\whB$ in~$k$,
let~$g(\zeta)$ and~$\whg(\zeta)$ be power series in~$k[[\zeta]]$ satisfying
$$ g(\zeta)
\equiv
\gamma \zeta \left( 1 + A \zeta^q + B\zeta^{2q} \right)
\mod \left\langle \zeta^{2q + 2} \right\rangle $$
and
$$ \whg(\zeta) \equiv \gamma \zeta \left( 1 + \whA \zeta^q + \whB
  \zeta^{2q} \right)
\mod \left\langle \zeta^{2q + 2} \right\rangle $$
in~$k[[\zeta]]$.
If~$g(\zeta)$ and~$\whg(\zeta)$ are conjugated, then we have
$$ \frac{B}{A^2} = \frac{\whB}{\whA^2}. $$
\end{lemm}
\begin{proof}
Let~$\lambda$ be in~$k^*$ and let~$h(\zeta)$ be a power series in~$k[[\zeta]]$ of the form
$$ h(\zeta)
=
\lambda \zeta \left( 1 + \beta_1 \zeta + \beta_2 \zeta^2 +
\cdots \right), $$
such that~$\whg(\zeta) = h \circ g \circ h^{-1}(\zeta)$.
Then we have
\begin{equation*}
h \circ g (\zeta)
\equiv
\lambda \gamma \zeta \left(1 + \beta_1 \gamma \zeta + \cdots +
  \beta_{q - 1} \gamma^{q - 1} \zeta^{q - 1} \right)
\mod \left\langle \zeta^{q + 1} \right\rangle,
\end{equation*}
and on the other hand
\begin{equation*}
\whg \circ h(\zeta)
\equiv
\gamma \lambda \zeta \left( 1 + \beta_1 \zeta + \cdots + \beta_{q - 1}
  \zeta^{q - 1} \right)
\mod \left\langle \zeta^{q + 1} \right\rangle.
\end{equation*}
Comparing coefficients we obtain
$$ \beta_1 = \cdots = \beta_{q - 1} = 0. $$
Therefore we have
\begin{multline*}
  \begin{aligned}
h \circ g (\zeta)
& \equiv
\lambda \gamma \zeta \left( 1 + A \zeta^q \right)
\\ & \quad \cdot
\left(1 + \beta_q \zeta^q + \beta_{q + 1} \gamma \zeta^{q + 1} +
 \cdots + \beta_{2q - 1} \gamma^{q - 1} \zeta^{2q - 1} \right)
\mod \left\langle \zeta^{2q + 1} \right\rangle
\\ & \equiv
\lambda \gamma \zeta
\left(1 + (A + \beta_q) \zeta^q + \beta_{q + 1} \gamma \zeta^{q + 1} +
  \cdots + \beta_{2q - 1} \gamma^{q - 1} \zeta^{2q - 1} \right)
  \end{aligned}
\\
\mod \left\langle \zeta^{2q + 1} \right\rangle.
\end{multline*}
On the other hand
\begin{multline*}
  \begin{aligned}
\whg \circ h(\zeta)
& \equiv
\gamma \lambda \zeta \left( 1 + \beta_q \zeta^q + \cdots + \beta_{2q - 1}
  \zeta^{2q - 1} \right)\left(1 + \whA \lambda^q \zeta^q \right)
\mod \left\langle \zeta^{2q + 1} \right\rangle
\\ & \equiv
\gamma \lambda \zeta \left( 1 + \left(\beta_q + \whA \lambda^q \right)
  \zeta^q + \beta_{q + 1} \zeta^{q + 1} + \cdots + \beta_{2q - 1}
  \zeta^{2q - 1} \right)
  \end{aligned}
\\
\mod \left\langle \zeta^{2q + 1} \right\rangle.
\end{multline*}
Comparing coefficients we obtain
$$ A = \whA \lambda^q
\text{ and }
\beta_{q + 1} = \cdots = \beta_{2q - 1} = 0. $$
In particular we have
$$ h(\zeta)
\equiv
\lambda \zeta \left( 1 + \beta_q \zeta^q + \beta_{2q} \zeta^{2q} \right)
\mod \left\langle \zeta^{2q + 2} \right\rangle. $$
Thus
\begin{equation*}
  \begin{split}
h \circ g(\zeta)
& \equiv
\lambda \gamma \zeta \left( 1 + A \zeta^q + B \zeta^{2q} \right)
\left(1 + \beta_q \zeta^q \left( 1 + A \zeta^q \right)^q + \beta_{2q}
  \zeta^{2q} \right)
\mod \left\langle \zeta^{2q + 2} \right\rangle
\\ & \equiv
\lambda \gamma \zeta \left( 1 + A \zeta^q + B \zeta^{2q} \right)
\left(1 + \beta_q \zeta^q + \left( q \beta_q A + \beta_{2q} \right) \zeta^{2q} \right)
\mod \left\langle \zeta^{2q + 2} \right\rangle
\\ & \equiv
\lambda \gamma \zeta
\left(1 + \left( A + \beta_q \right) \zeta^q + \left( B + (q + 1) \beta_q A + \beta_{2q} \right) \zeta^{2q} \right)
\mod \left\langle \zeta^{2q + 2} \right\rangle,
  \end{split}
\end{equation*}
and on the other hand
\begin{equation*}
  \begin{split}
\whg \circ h (\zeta)
& \equiv
\gamma \lambda \zeta
\left(1 + \beta_q \zeta^q + \beta_{2q} \zeta^{2q} \right)
\left( 1 + \whA \lambda^q \zeta^q \left(1 + \beta_q \zeta^q\right)^q +
  \whB \lambda^{2q} \zeta^{2q} \right)
\\ & \quad
\mod \left\langle \zeta^{2q + 2} \right\rangle
\\ & \equiv
\gamma \lambda \zeta
\left(1 + \beta_q \zeta^q + \beta_{2q} \zeta^{2q} \right)
\left( 1 + \whA \lambda^q \zeta^q + \left(q \whA \lambda^q \beta_q +
    \whB \lambda^{2q} \right) \zeta^{2q} \right)
\\ & \quad
\mod \left\langle \zeta^{2q + 2} \right\rangle
\\ & \equiv
\gamma \lambda \zeta
\left( 1 + \left( \beta_q + \whA \lambda^q \right) \zeta^q + \left( \beta_{2q} + (q + 1) \whA \lambda^q \beta_q + \whB \lambda^{2q} \right) \zeta^{2q} \right)
\\ & \quad
\mod \left\langle \zeta^{2q + 2} \right\rangle
  \end{split}
\end{equation*}
Comparing coefficients and using~$\lambda^q = A / \whA$ we obtain the lemma.
\end{proof}

\begin{proof}[Proof of Proposition~\ref{p:basic normal form}]
The last assertion is given by Lemma~\ref{l:invariance of resit}.
Since~$i_0(\whg^q) = i_0(g^q)$, the equivalence between~$a_1 \neq 0$
and~$i_0(g^q) = q$ is trivial when~$q = 1$, and it
follows from Lemma~\ref{l:q iterate} with~$\ell = q$ when~$q \ge 2$.

It remains to prove the first assertion of the proposition.
In the case~$q = 1$, take~$\whg = g$.
Assume~$q \ge 2$.
Let~$s_0(\zeta)$ and~$h_0(\zeta)$ be the power series
in~$k[[\zeta]]$ defined by
$$ s_0(\zeta) \= 1
\text{ and }
h_0(\zeta) \= \zeta. $$
We define inductively for every integer~$\ell \ge 1$ polynomials~$s_{\ell}(\zeta)$ and~$h_{\ell}(\zeta)$ in~$k[\zeta]$ of
degree at most~$\ell + 1$ and~$\left[ \frac{\ell}{q} \right]$, respectively, such that
$$ h_{\ell}(\zeta)
\equiv
h_{\ell - 1}(\zeta) \mod \left\langle \zeta^{\ell + 1} \right\rangle
\text{ and }
s_{\ell}(\zeta)
\equiv
s_{\ell - 1}(\zeta) \mod \left\langle \zeta^{\left[ \frac{\ell - 1}{q} \right]} \right\rangle, $$
and such that the power series~$g_{\ell}(\zeta) \= h_{\ell} \circ g \circ
h_{\ell}^{-1}(\zeta)$ in~$k[[\zeta]]$ satisfies
\begin{equation}
  \label{e:finite level clearing}
g_{\ell}(\zeta)
\equiv
\gamma \zeta s_{\ell}(\zeta^q) \mod \left\langle \zeta^{\ell + 2} \right\rangle.  
\end{equation}
This clearly implies the first assertion of the proposition.

Note that
$$ g_0(\zeta)
\=
h_0 \circ g \circ h_0^{-1}(\zeta)
=
g(\zeta)
\equiv
\gamma \zeta \mod \left\langle \zeta^2 \right\rangle, $$
so~\eqref{e:finite level clearing} is satisfied when~$\ell = 0$.
Let~$\ell \ge 1$ be an integer for which~$s_{\ell - 1}(\zeta)$
and~$h_{\ell - 1}(\zeta)$ are already defined and
satisfy~\eqref{e:finite level clearing} with~$\ell$ replaced by~$\ell
- 1$, and let~$A$ in~$k$ be such that
$$ g_{\ell - 1}(\zeta)
\equiv
\gamma \zeta \left( s_{\ell - 1}(\zeta^q) + A \zeta^\ell \right) \mod
\left\langle \zeta^{\ell + 2} \right\rangle. $$
In the case~$\ell$ is divisible by~$q$, the congruence~\eqref{e:finite level
  clearing} is verified if we put
$$ h_{\ell}(\zeta) \= h_{\ell - 1}(\zeta)
\text{ and }
s_{\ell}(\zeta) \= s_{\ell - 1}(\zeta) + A \zeta^{\frac{\ell}{q}}. $$
Suppose~$\ell$ is not divisible by~$q$, put~$\alpha \= -
\frac{A}{\gamma^\ell - 1}$, and define
$$ h(\zeta) \= \zeta \left( 1 + \alpha \zeta^\ell \right). $$
Moreover, put
$$ h_{\ell}(\zeta) \= h \circ h_{\ell - 1}(\zeta)
\text{ and }
s_{\ell}(\zeta) \= s_{\ell - 1}(\zeta). $$
Then we have
$$ g_{\ell}(\zeta)
=
h_\ell \circ g \circ h_{\ell}^{-1}(\zeta)
=
h \circ g_{\ell - 1} \circ h^{-1}(\zeta), $$
so there is~$B$ in~$k$ such that
\begin{equation*}
  \begin{split}
g_{\ell}(\zeta)
& \equiv
\gamma \zeta \left( s_{\ell - 1}(\zeta^q) + B \zeta^\ell \right)
\mod \left\langle \zeta^{\ell + 2} \right\rangle
\\ & \equiv
\gamma \zeta \left( s_{\ell}(\zeta^q) + B \zeta^\ell \right)
\mod \left\langle \zeta^{\ell + 2} \right\rangle.
  \end{split}
\end{equation*}
Thus, to complete the proof of the induction step it is enough to show that~$B = 0$.
To do this, note that by our definition of~$\alpha$ we have
\begin{equation*}
  \begin{split}
h \circ g_{\ell - 1}(\zeta)
& =
g_{\ell - 1}(\zeta) \left( 1 + \alpha g_{\ell - 1}(\zeta)^\ell \right)
\\ & \equiv
\gamma \zeta \left( s_{\ell - 1}(\zeta^q) + A \zeta^\ell \right)
\left( 1 + \alpha \gamma^{\ell} \zeta^{\ell} \right)
\mod \left\langle \zeta^{\ell + 2} \right\rangle
\\ & \equiv
\gamma \zeta \left( s_{\ell - 1}(\zeta^q) + \left( A + \alpha \gamma^{\ell} \right) \zeta^{\ell} \right)
\mod \left\langle \zeta^{\ell + 2} \right\rangle
\\ & \equiv
\gamma \zeta \left( s_{\ell - 1}(\zeta^q) + \alpha \zeta^{\ell} \right)
\mod \left\langle \zeta^{\ell + 2} \right\rangle.
  \end{split}
\end{equation*}
On the other hand
\begin{equation*}
  \begin{split}
g_{\ell} \circ h(\zeta)
& \equiv
\gamma h(\zeta) \left( s_{\ell - 1} \left( h(\zeta)^q \right) + B h(\zeta)^\ell \right)
\mod \left\langle \zeta^{\ell + 2} \right\rangle
\\ & \equiv
\gamma \zeta \left( 1 + \alpha \zeta^{\ell} \right)
\left( s_{\ell - 1} \left( \zeta^q \right) + B \zeta^{\ell} \right)
\mod \left\langle \zeta^{\ell + 2} \right\rangle
\\ & \equiv
\gamma \zeta \left( s_{\ell - 1} \left( \zeta^q \right) + (\alpha + B) \zeta^{\ell} \right)
\mod \left\langle \zeta^{\ell + 2} \right\rangle.
  \end{split}
\end{equation*}
Comparing coefficients we conclude that~$B = 0$.
This completes the proof of the induction step and of the proposition.
\end{proof}

\subsection{Proof of Theorem~\ref{t:characterization of minimally ramified}}
\label{ss:proof of characterization of minimally ramified}
The proof of Theorem~\ref{t:characterization of minimally ramified} is
given after the following proposition, which is the special case~$q = 1$.

When~$p$ is odd, the following proposition is~\cite[\emph{Exemple}~$3.19$]{Rivthese}.
We include its short proof for completeness.
\begin{prop}
\label{p:concrete minimally ramified}
Let~$p$ be a prime number and $k$ a field of characteristic~$p$.
Given~$a_1$ and~$a_2$ in~$k$, let~$g(\zeta)$ be a power series in~$k[[\zeta]]$ satisfying
$$ g(\zeta)
\equiv
\zeta \left( 1 + a_1 \zeta + a_2 \zeta^2 \right)
\mod \left\langle \zeta^4 \right\rangle $$
in~$k[[\zeta]]$.
If~$p$ is odd (resp.~$p = 2$), then~$g$ is minimally ramified if and only if
$$ a_1 \neq 0
\text{ and }
a_2 \neq a_1^2
\quad
\left( \text{resp. } a_1 \neq 0,
a_2 \neq 0,
\text{ and }
a_2 \neq a_1^2 \right). $$
\end{prop}
\begin{proof}
Note that~$i_0(g) = 1$ is equivalent to~$a_1 \neq 0$.
Since this is necessary for~$g$ to be minimally ramified, we
assume~$a_1 \neq 0$.
In part~$1$ we prove the proposition when~$p$ is odd and in part~$2$
when~$p = 2$.

\partn{1}
Suppose~$p$ is odd.
Note that by Proposition~\ref{p:minimally ramified} the power series~$g$ is minimally ramified if and only if~$i_1(g) = p + 1$.

For~$n$ in~$\{1, \ldots, p \}$ define the power series~$\Delta_n(\zeta)$ in~$k[[\zeta]]$ inductively by~$\Delta_1(\zeta) \= g(\zeta) - \zeta$, and for~$n$ in~$\{2, \ldots, p \}$ by
$$ \Delta_n(\zeta)
\=
\Delta_{n - 1}(g(\zeta)) - \Delta_{n - 1}(\zeta). $$
Note that~$\Delta_p(\zeta) = g^p(\zeta) - \zeta$.

We prove first
\begin{equation}
  \label{e:just before jump}
  \Delta_{p - 1}(\zeta)
\equiv
- a_1^{p - 1} \zeta^{p} - a_1^{p - 2} \left( a_2 - a_1^2 \right)
\zeta^{p + 1}
\mod \left\langle \zeta^{p + 2} \right\rangle.
\end{equation}
To do this, first we prove by induction that for every~$n$ in~$\{1,
\ldots, p - 1 \}$ we have
\begin{equation}
  \label{e:before jump}
  \Delta_n(\zeta)
\equiv
n! a_1^n \zeta^{n + 1} \mod \left\langle \zeta^{n + 2} \right\rangle.
\end{equation}
When~$n = 1$ this is true by the definition of~$\Delta_1$.
Let~$n$ in~$\{1, \ldots, p - 2 \}$ be such that~\eqref{e:before jump} holds, and let~$B$ in~$k$ be such that
$$ \Delta_n(\zeta)
\equiv
n! a_1^n \zeta^{n + 1} + B \zeta^{n + 2}
\mod \left\langle \zeta^{n + 3} \right\rangle. $$
Then
\begin{equation*}
  \begin{split}
    \Delta_{n + 1}(\zeta)
& \equiv
n! a_1^n \zeta^{n + 1} \left( 1 + a_1 \zeta \right)^{n + 1} + B \zeta^{n + 2}
\\ & \quad
- \left( n! a_1^n \zeta^{n + 1} + B \zeta^{n + 2} \right)
\mod \left\langle \zeta^{n + 3} \right\rangle.
\\ & \equiv
(n + 1)! a_1^{n + 1} \zeta^{n + 2}
\mod \left\langle \zeta^{n + 3} \right\rangle.
  \end{split}
\end{equation*}
This proves the induction step and~\eqref{e:before jump}.
To prove~\eqref{e:just before jump}, put~$A' \= a_1^{p - 2}$ and note that by~\eqref{e:before jump} with~$n = p - 2$ there are~$B'$ and~$C'$ in~$k$ such that
$$ \Delta_{p - 2}(\zeta)
\equiv
A' \zeta^{p - 1} + B' \zeta^p + C' \zeta^{p + 1}
\mod \left\langle \zeta^{p + 2} \right\rangle. $$
We thus have
\begin{equation*}
  \begin{split}
\Delta_{p - 1}(\zeta)
& \equiv
A' \zeta^{p - 1} \left( 1 + a_1 \zeta + a_2 \zeta^2 \right)^{p - 1}
+ B' \zeta^p(1 + a_1 \zeta)^p
+ C' \zeta^{p + 1}
\\ & \quad
- \left( A' \zeta^{p - 1} + B' \zeta^p + C' \zeta^{p + 1} \right)
\mod \left\langle \zeta^{p + 2} \right\rangle
\\ & \equiv
- A' a_1 \zeta^p + A' \left( a_1^2 - a_2 \right) \zeta^{p + 1}
\mod \left\langle \zeta^{p + 2} \right\rangle.
  \end{split}
\end{equation*}
This proves~\eqref{e:just before jump}.

To complete the proof, put
$$ A'' \= - a_1^{p - 1}
\text{ and }
B'' \= - a_1^{p - 2} \left( a_2 - a_1^2 \right), $$
and note that by~\eqref{e:just before jump} there is~$C''$ in~$k$ such that
\begin{equation*}
\Delta_{p - 1}(\zeta)
\equiv
A'' \zeta^p + B'' \zeta^{p + 1} + C'' \zeta^{p + 2}
\mod \left\langle \zeta^{p + 3} \right\rangle.
\end{equation*}
Using~$p \ge 3$, we have
\begin{equation*}
  \begin{split}
\Delta_p(\zeta)
& \equiv
A'' \zeta^p \left( 1 + a_1 \zeta + a_2 \zeta^2 \right)^p
+ B'' \zeta^{p + 1} \left( 1 + a_1 \zeta \right)^{p + 1}
+ C'' \zeta^{p + 2}
\\ & \quad
-\left( A'' \zeta^p + B'' \zeta^{p + 1} + C'' \zeta^{p + 2} \right)
\mod \left\langle \zeta^{p + 3} \right\rangle
\\ & \equiv
a_1 B'' \zeta^{p + 2}
\mod \left\langle \zeta^{p + 3} \right\rangle.
\\ & \equiv
- a_1^{p - 1}\left( a_2 - a_1^2 \right) \zeta^{p + 2}
\mod \left\langle \zeta^{p + 3} \right\rangle.
\end{split}
\end{equation*}
This proves that we have~$i_1(g) = p + 1$ if and only if~$a_2 \neq a_1^2$, and
completes the proof of the proposition when~$p$ is odd.

\partn{2}
Suppose~$p = 2$.
Note that by Proposition~\ref{p:minimally ramified} the power series~$g$ is
minimally ramified if and only if~$i_2(g^q) = 7$.

Put
$$ \Delta_1(\zeta) \= g(\zeta) - \zeta
\text{ and }
\Delta_2(\zeta) \= \Delta_1 (g(\zeta)) - \Delta_1(\zeta), $$
and note that~$\Delta_2(\zeta) = g^2(\zeta) - \zeta$.
Letting~$a_3$ and~$a_4$ in~$k$ be such that
$$ g(\zeta)
\equiv
\zeta \left( 1 + a_1 \zeta + a_2 \zeta^2 + a_3 \zeta^3 + a_4
  \zeta^4 \right)
\mod \left\langle \zeta^6 \right\rangle, $$
we have
\begin{equation*}
  \begin{split}
    \Delta_2(\zeta)
& \equiv
a_1 \zeta^2 \left( 1 + a_1 \zeta + a_2 \zeta^2 + a_3 \zeta^3 \right)^2
+ a_2 \zeta^3 \left(1 + a_1 \zeta + a_2 \zeta^2 \right)^3
\\ & \quad
+ a_3 \zeta^4 \left( 1 + a_1 \zeta \right)^4
+ a_4 \zeta^5
\\ & \quad
- \left( a_1 \zeta^2 + a_2 \zeta^3 + a_3 \zeta^4 + a_4 \zeta^5 \right)
\mod \left\langle \zeta^6 \right\rangle
\\ & \equiv
a_1 \zeta^2 \left( 1 + a_1^2 \zeta^2 \right)
+ a_2 \zeta^3 \left( 1 + a_1 \zeta + \left( a_2 + a_1^2 \right) \zeta^2 \right)
\\ & \quad
- \left( a_1 \zeta^2 + a_2 \zeta^3 \right)
\mod \left\langle \zeta^6 \right\rangle
\\ & \equiv
a_1 \left( a_2 + a_1^2 \right) \zeta^4
+ a_2 \left( a_2 + a_1^2 \right) \zeta^5 
\mod \left\langle \zeta^6 \right\rangle.
\end{split}
  \end{equation*}
We thus obtain
\begin{equation}
\label{e:second iterate}
g^2(\zeta)
\equiv
\zeta \left( 1 + a_1 \left( a_2 + a_1^2 \right) \zeta^3
+ a_2 \left( a_2 + a_1^2 \right) \zeta^4 \right)
\mod \left\langle \zeta^6 \right\rangle.
\end{equation}

Put
$$ \hDelta_1(\zeta) \= g^2(\zeta) - \zeta
\text{ and }
\hDelta_2(\zeta) \= \hDelta_1 \left( g^2(\zeta) \right) - \hDelta_1(\zeta), $$
and note that~$\hDelta_1(\zeta) = \Delta_2(\zeta)$ and~$\hDelta_2(\zeta) = g^4(\zeta) - \zeta$.
So, if we put
$$ A \= a_1 \left( a_2 + a_1^2 \right)
\text{ and }
B \= a_2 \left( a_2 + a_1^2 \right), $$
then by~\eqref{e:second iterate} there are~$C$, $D$, and~$E$ in~$k$ such that
$$ \hDelta_1(\zeta)
\equiv
A \zeta^4 + B \zeta^5 + C \zeta^6 + D \zeta^7 + E \zeta^8
\mod \left\langle \zeta^9 \right\rangle. $$
Then we have
\begin{equation*}
  \begin{split}
\hDelta_2(\zeta)
& \equiv
A \zeta^4 \left( 1 + A \zeta^3 + B \zeta^4 \right)^4
+ B \zeta^5 \left( 1 + A \zeta^3 \right)^5
+ C \zeta^6 + D \zeta^7 + E \zeta^8
\\ & \quad
- \left( A \zeta^4 + B \zeta^5 + C \zeta^6 + D \zeta^7 + E \zeta^8 \right)
\mod \left\langle \zeta^9 \right\rangle
\\ & \equiv
A B \zeta^8
\mod \left\langle \zeta^9 \right\rangle
\\ & \equiv
a_1 a_2 \left( a_2 + a_1^2 \right)^2 \zeta^8
\mod \left\langle \zeta^9 \right\rangle.
  \end{split}
\end{equation*}
This proves that~$i_2(g^q) = 7$ if and only if~$a_2 \neq 0$ and~$a_2
\neq a_1^2$, and completes the proof of the proposition when~$p = 2$.
\end{proof}

\begin{proof}[Proof of Theorem~\ref{t:characterization of minimally ramified}]
Since minimal ramification is invariant under conjugacy, by Proposition~\ref{p:basic normal form} we can assume that~$g(\zeta)$ is of the form
$$ g(\zeta)
=
\gamma \zeta \left(1 + \sum_{j = 1}^{+ \infty} a_j \zeta^{jq}
\right). $$
By Proposition~\ref{p:basic normal form} when~$q \ge 2$, we have~$a_1 \neq 0$ if and only
if~$i_0(g^q) = q$.
Since this last condition is necessary for~$g$ to be minimally
ramified, from now on we assume~$a_1 \neq 0$.

Put
$$ \pi(\zeta) \= \zeta^q
\text{ and }
\whg(\zeta) \= \zeta \left( 1 + \sum_{j = 1}^{+ \infty} a_j \zeta^j \right)^q, $$
and note that~$\pi \circ g = \whg \circ \pi$.
Since~$q$ is not divisible by~$p$, this implies that for every
integer~$n \ge 1$ we have
$$ i_n(g)
=
\ord \left( \left( \frac{g^{qp^n}(\zeta)}{\zeta} \right)^q - 1 \right)
=
\ord \left( \frac{\whg^{qp^n} \circ \pi(\zeta) -
    \pi(\zeta)}{\pi(\zeta)} \right)
=
q i_n(\whg^q)
=
q i_n(\whg). $$
Thus~$g$ is minimally ramified if and only if~$\whg$ is.
Then the theorem is a direct consequence of Proposition~\ref{p:concrete minimally ramified}.
\end{proof}

\section{Optimal cycles}
\label{s:optimality}
In this section we address the optimality of the periodic points lower bounds~\eqref{e:periodic bound} and~\eqref{e:fixed bound}.
In~\S\ref{ss:generic optimality} we prove a general version of Theorem~\ref{t:optimality and minimality}, that we state as Theorem~\ref{t:optimality and minimality}'.
This result implies in particular that the lower bound~\eqref{e:periodic bound} is optimal.
In~\S\ref{ss:concrete optimality} we exhibit concrete polynomials that satisfy the conclusions of Theorem~\ref{t:periodic optimality}, see Propositions~\ref{p:odd optimality} and~\ref{p:even periodic optimality}.
Part~$1$ of Proposition~\ref{p:odd optimality} implies that inequality~\eqref{e:fixed bound} is optimal.
The proof of Theorem~\ref{t:quadratic optimality and minimality} is given at the end of~\S\ref{ss:concrete optimality}.

\subsection{Optimality of the Periodic Points Lower Bound}
\label{ss:generic optimality}
The purpose of this section is to prove the following result, which is a more general version of Theorem~\ref{t:optimality and minimality}.

\begin{generic}[Theorem~\ref{t:optimality and minimality}']
Let~$p$ be a prime number, $(K, |\cdot|)$ an algebraically closed field of residue characteristic~$p$, and~$q \ge 1$ an integer that is not divisible by~$p$.
Then the following properties hold.
\begin{enumerate}
\item[1.]
Let~$\lambda$ in~$K$ be such that~$|\lambda| = 1$ and such that the order of~$\tlambda$ in~$\tK^*$ is~$q$.
Moreover, let $n \ge 1$ be an integer and~$P(z) = \lambda z + \cdots $ a polynomial in~$\OK[z]$ having an optimal cycle of period~$qp^n$.
Then this is the only cycle of minimal period~$qp^n$ of~$f$, and if~$p$ is odd, then the reduction of~$P$ is minimally ramified.
If~$p = 2$ and~$n \ge 2$, then the reduction of~$P$ is minimally ramified or almost minimally ramified.
\item[2.]
Let~$F$ be the prime field of~$K$, $\gamma$ a root of unity in~$\tK$ of order~$q$, and~$g(\zeta) = \gamma \zeta + \cdots$ a polynomial in~$\tK[\zeta]$ that is either minimally ramified if~$p$ is odd, or minimally ramified or almost minimally ramified if~$p = 2$.
Then for all integers~$n \ge 1$ and~$d \ge \max\{\deg(g), p\}$, there is a nonzero polynomial~$R_n(\alpha_1, \ldots, \alpha_d)$ in~$F[\alpha_1, \ldots, \alpha_d]$ such that the following property holds.
If~$a_1$, \ldots, $a_d$ in~$\OK$ are such that the reduction of the polynomial~$P(z) \= a_1 z + \cdots + a_d z^d$ is~$g$ and such that~$R_n(a_1, \ldots, a_n) \neq 0$, then~$P$ has an optimal cycle of period~$qp^n$.
\end{enumerate}
\end{generic}

The proof of Theorem~\ref{t:optimality and minimality}' is at the end of this section.
We use the following general criterion, which is stated in a more general form than what is needed for this section.

\begin{prop}
\label{p:detailed optimality criterion}
Let~$p$ be a prime number, and~$(K, |\cdot|)$ an algebraically closed ultrametric field of residue characteristic~$p$.
Given an integer~$q \ge 1$ that is not divisible by~$p$, let~$\lambda$ in~$K$ be such that~$|\lambda| = 1$, and such that the order of~$\tlambda$ in~$\tK^*$ is equal to~$q$.
Then, for every power series~$f(z) = \lambda z + \cdots$ in~$\OK[[z]]$ the following properties hold.
\begin{enumerate}
\item[1.]
Suppose~$\wideg(f^q(z) - z) = q + 1$ and~$\lambda^q \neq 1$.
Then~$f$ has a unique periodic orbit in~$\mK \setminus \{ 0 \}$ of minimal period~$q$, and for every periodic point~$w_0$ in this orbit, inequality~\eqref{e:fixed bound} holds with equality.
\item[2.]
Let~$n \ge 1$ be an integer, and suppose that
$$ \wideg \left( \frac{f^{qp^n}(z) - z}{f^{qp^{n - 1}}(z) - z} \right)
=
qp^n, $$
and that for every periodic point~$z_0$ of period~$qp^{n - 1}$ we have~$(f^{qp^n})'(z_0) \neq 1$.
Then there is a unique cycle of~$f$ of minimal period~$qp^n$, and this cycle is optimal.
\end{enumerate}
\end{prop}

Note that to formulate part~$2$, we used the fact that the power series~$f^{qp^{n - 1}}(z) - z$ divides~$f^{qp^n}(z) - z$ in~$\OK[[z]]$; this is obtained by applying Lemma~\ref{l:fixed are periodic} with~$g = f^{qp^{n - 1}}$ and~$m = p$.

\begin{proof}[Proof of Proposition~\ref{p:detailed optimality criterion}]
To prove part~$1$, note first that every periodic point of~$f$ in~$\mK \setminus \{ 0 \}$ of minimal period~$q$ is a zero of~$(f^q(z) - z)/z$.
Thus, our hypothesis~$\wideg(f^q(z) - z) = q + 1$ implies that there is at most~$1$ periodic orbit of~$f$ in~$\mK \setminus \{ 0 \}$ of minimal period~$q$.
To prove that such a periodic orbit exists, note that our hypotheses~$\lambda^q \neq 1$ and~$\wideg(f^q(z) - z) = q + 1$, imply that there is at least one zero~$w_0$ of~$(f^q(z) - z)/z$ in~$\mK \setminus \{ 0 \}$.
Then~$w_0$ is a periodic point of~$f$ of period~$q$, and Lemma~\ref{l:admissible periods} implies that the minimal period of~$w_0$ is~$q$.
This proves that there is a unique periodic orbit of~$f$ in~$\mK \setminus \{ 0 \}$ of period~$q$.
In view of our hypothesis~$\wideg(f^q(z) - z) = q + 1$, part~$1$ of Lemma~\ref{l:periodic bound} implies that~\eqref{e:fixed bound} holds with equality.
This completes the proof of part~$1$.

To prove part~$2$, put
$$ h(z) \= \frac{f^{qp^n}(z) - z}{f^{qp^{n - 1}}(z) - z}, $$
and note that every periodic point of~$f$ of minimal period~$qp^n$ is a zero of~$h$.
Thus, our hypothesis~$\wideg(h) = qp^n$ implies that there is at most~$1$ periodic orbit of~$f$ of minimal period~$qp^n$.
To prove that such a periodic orbit exist, note that our hypothesis~$\wideg(h) = qp^n$ implies that~$h$ has a zero~$z_0$ in~$\mK$.
Then~$z_0$ is also a zero of~$f^{qp^n}(z) - z$, and therefore~$z_0$ is a periodic point of~$f$ of period~$qp^n$.
Suppose by contradiction that the minimal period of~$z_0$ is not~$qp^n$.
Then Lemma~\ref{l:admissible periods} implies that~$z_0$ is of period~$qp^{n - 1}$, and therefore a zero of~$f^{qp^{n - 1}}(z) - z$.
By hypothesis we also have~$(f^{qp^n})'(z_0) \neq 1$.
On the other hand, since~$z_0$ is also a zero of~$h$, it follows that~$z_0$ is a multiple zero of~$f^{qp^n}(z) - z$.
This implies that~$z_0$ is also a zero of~$(f^{qp^n})'(z) - 1$, so~$(f^{qn})'(z_0) = 1$.
We thus obtain a contradiction that shows that the minimal period of~$z_0$ is~$qp^n$, and that there is a unique periodic orbit of~$f$ of minimal period~$qp^n$.
Finally, note that our hypothesis~$\wideg(h) = qp^n$, together with part~$2$ of Lemma~\ref{l:periodic bound}, imply that~\eqref{e:fixed bound} holds with equality.
This completes the proof of part~$2$, and of the lemma.
\end{proof}

\begin{lemm}
\label{l:generic independence}
Let~$K$ be a field, $d \ge 2$ an integer, and~$a_1$, \ldots, $a_d$ in~$K$ algebraically independent over the prime field of~$K$.
If the characteristic~$p$ of~$K$ is positive, suppose in addition that~$d \ge p$.
Then the polynomial
$$ a_1 z + \cdots + a_d z^d $$ 
in~$K[z]$ has no parabolic periodic point.
\end{lemm}
\begin{proof}
When the characteristic of~$K$ is zero, the desired assertion is Lemma~\ref{l:generic independence zero characteristic}.
Suppose the characteristic~$p$ of~$K$ is positive and that~$d \ge p$.
Denote by~$\F_p$ the prime field of~$K$, and by~$\overline{\F}_p$ an algebraic closure of~$\F_p$.

Suppose by contradiction there is an integer~$n \ge 1$ and a periodic point~$z_0$ of period~$n$ of the polynomial~$P(z) \= a_1 z + \cdots + a_d z^d$ in~$K[z]$, such that~$(P^n)'(z_0)$ is a root of unity.
Let~$\sigma : \F_p[z_0, a_1, \ldots, a_d] \to \overline{\F}_p$ be a ring homomorphism such that~$\sigma(a_p) = 1$, and such that for each~$j$ in~$\{1, \ldots, d \}$ different from~$p$ we have~$\sigma(a_j) = 0$.
Then~$\sigma(P)(z) = z^p$, $\sigma(z_0)$ is a periodic point of period~$n$ of~$\sigma(P)$, and~$(\sigma(P)^n)'(\sigma(z_0)) = \sigma((P^n)'(z_0))$ is a root of unity.
On the other hand, $\sigma(P)'$ is the zero polynomial, so~$(\sigma(P)^n)'(\sigma(z_0)) = 0$.
This contradiction completes the proof of the lemma.
\end{proof}

\begin{proof}[Proof of Theorem~\ref{t:optimality and minimality}']
The uniqueness statement in part~$1$ is given by part~$2$ of Lemma~\ref{l:periodic bound}.
The rest of the assertions of part~$1$ are given by Corollary~\ref{c:optimality implies minimality}.

To prove part~$2$, let~$\alpha_0$, \ldots, $\alpha_d$ be algebraically independent over~$F$, and consider the polynomial
$$ Q(z) = \alpha_1 z + \cdots + \alpha_d z^d $$
in~$F[\alpha_1, \ldots, \alpha_d][z]$.
Let~$R_n(\alpha_0, \ldots, \alpha_d)$ in~$F[\alpha_0, \ldots, \alpha_d]$ be the resultant of the polynomials
$$ Q^{qp^{n - 1}}(z) - z
\text{ and }
\left( Q^{qp^n} \right)'(z) - 1. $$
Lemma~\ref{l:generic independence} with~$P = Q$ implies that~$R_n(\alpha_0, \ldots, \alpha_d)$ is nonzero.
If~$n$ and~$P$ are as in the statement of part~$2$ of the theorem, then
$$ \wideg \left( \frac{P^{qp^n}(z) - z}{P^{qp^{n - 1}}(z) - z} \right)
=
i_n(g^q) - i_{n - 1}(g^q)
=
qp^n, $$
and for every periodic point~$z_0$ of~$P$ of period~$qp^{n - 1}$ we
have~$\left( P^{qp^n} \right)'(z_0) \neq 1$.
Then by part~$2$ of Proposition~\ref{p:detailed optimality criterion} with~$f = P$, the polynomial~$P$ has an optimal cycle of period~$qp^n$.
This completes the proof.
\end{proof}

\subsection{Concrete polynomials having optimal cycles}
\label{ss:concrete optimality}
The purpose of this section is to exhibit concrete polynomials having optimal cycles.
The case where~$p$ is odd is covered by Proposition~\ref{p:odd optimality}, and the case~$p = 2$ by Proposition~\ref{p:even periodic optimality}.
Theorem~\ref{t:periodic optimality} is a direct consequence of these propositions.

The proof of Theorem~\ref{t:quadratic optimality and minimality} is given at the end of this section.
\begin{prop}
\label{p:odd optimality}
Let~$p$ be a prime number, and~$(K, |\cdot|)$ an algebraically closed ultrametric field of residue characteristic~$p$.
Given an integer~$q \ge 1$ that is not divisible by~$p$, let~$\lambda$ in~$K$ be transcendental over the prime field of~$K$, such that~$|\lambda| = 1$, and such that the order of~$\tlambda$ in~$\tK^*$ is equal to~$q$.
Moreover, let~$P(z)$ be the polynomial in~$K[z]$ defined by
$$ P(z) \= \lambda z \left(1 + z^q \right), $$
if~$p$ does not divide~$q + 1$, and by
$$ P(z) \= \lambda z \left(1 + z^q + z^{2q} \right) $$
otherwise.
Then the following properties hold.
\begin{enumerate}
\item[1.]
There is a unique periodic orbit of~$P$ in~$\mK \setminus \{ 0 \}$ of minimal period~$q$, and for every periodic point~$w_0$ in this orbit, inequality~\eqref{e:fixed bound} holds with equality.
\item[2.]
If~$p$ is odd, then for every~$n \ge 1$ there is a unique periodic orbit of~$P$ of minimal period~$q p^n$.
Furthermore, for every point~$z_0$ in this orbit, inequality~\eqref{e:periodic bound} holds with equality.
\end{enumerate}
\end{prop}

\begin{rema}
\label{r:concrete optimality postive characteristic}
If~$K$ in the proposition above is of positive characteristic, then for~$\lambda$ in~$K$ such that~$|\lambda| = 1$ and such that the order of~$\tlambda$ in~$\tK^*$ is~$q$, the hypothesis that~$\lambda$ is transcendental over the prime field of~$K$ is equivalent to~$\lambda^q \neq 1$.
\end{rema}

\begin{rema}
In the case where~$p$ divides~$q + 1$, our results imply that the
conclusions of part~$2$ of Proposition~\ref{p:odd optimality} are
false for the polynomial $\lambda z \left(1 + z^q \right)$.
\end{rema}

\begin{prop}
\label{p:even periodic optimality}
Let~$(K, |\cdot|)$ be an algebraically closed ultrametric field of residue characteristic~$2$.
Given an odd integer~$q \ge 1$, let~$\lambda$ in~$K$ be transcendental over the prime field of~$K$, such that~$|\lambda| = 1$, and such that the order of~$\tlambda$ in~$\tK^*$ is~$q$.
In the case where the characteristic of~$K$ is zero, put
$$ Q(z) \= \lambda z \left( 1 + z^{2q} \right). $$
In the case the characteristic of~$K$ is~$2$, let~$\mu$ in~$\mK$ be algebraically independent with respect to~$\lambda$ over the prime field of~$K$, and put
$$ Q(z) \= \lambda z \left( 1 + \mu z^q + z^{2q} \right). $$
Then, for every integer~$n \ge 1$ there is a unique periodic orbit of~$Q$ of minimal period~$2^n q$.
Furthermore, for every periodic point~$z_0$ in this orbit, inequality~\eqref{e:periodic bound} holds with equality.
\end{prop}

\begin{rema}
\label{r:concrete optimality zero characteristic}
If~$K$ in either Proposition~\ref{p:odd optimality} or~\ref{p:even periodic optimality} is of characteristic zero, then~$\lambda$ can be allowed to be algebraic over the prime field of~$K$, as long as~$\lambda$ avoids a finite set of exceptional values that depends on~$n$.
Similarly, in the case~$K$ in Proposition~\ref{p:even periodic optimality} is of characteristic~$2$, we show that for each~$n$ there is a nonzero polynomial~$R_n$ in~$\lambda$ and~$\mu$ with coefficients in the prime field of~$K$, such that the conclusions of the proposition hold whenever~$R_n(\lambda, \mu)$ is nonzero.
Thus, $\lambda$ and~$\mu$ can be allowed to be algebraic over the prime field of~$K$, as long as~$(\lambda, \mu)$ avoids a curve in~$K \times K$ that depends on~$n$.
\end{rema}

\begin{rema}
\label{r:non-example in characteristic 2}
In Appendix~\ref{s:all minimal} we show that, when~$K$ is of characteristic~$2$, the conclusions of Proposition~\ref{p:even periodic optimality} are false if we let~$\mu = 0$, see Remark~\ref{r:non-example appendix}.
\end{rema}

The following lemma is the main ingredient in the proofs of Propositions~\ref{p:odd optimality} and~\ref{p:even periodic optimality}, and of Theorem~\ref{t:quadratic optimality and minimality}.
\begin{lemm}
\label{l:concrete independence}
Let~$K$ be a field, and $\lambda$ in~$K$ that is transcendental over the prime field of~$K$.
Moreover, let~$q \ge 1$ be an integer, and let~$P(z)$ be the polynomial in~$K[z]$ defined by either
$$ P(z) \= \lambda z \left( 1 + z^q \right),
\text{ or }
P(z) \= \lambda z \left( 1 + z^q + z^{2q} \right). $$
If the characteristic~$p$ of~$K$ is positive, suppose in addition that~$p$ does not divide~$q$.
Then $P(z)$ has no parabolic periodic point.
\end{lemm}
\begin{proof}
Let~$F$ be the prime field of~$K$.
Without loss of generality assume that~$K$ is an algebraic closure of the field~$F(\lambda)$.

\partn{Case 1}
The characteristic of~$K$ is zero.
We give the proof in the case~$P(z) = \lambda z \left( 1 + z^q +
  z^{2q} \right)$.
The proof in the case~$P(z) = \lambda z \left(1 + z^q \right)$ is analogous.\footnote{Note also that the proof in Case~$2.1$, stated for the case where the characteristic of~$K$ is positive, also works in the case the characteristic of~$K$ is zero.}
Given~$\alpha$ in~$K$, consider the polynomial~$Q_{\alpha}(z) \= \alpha^2 z + \alpha z^{q + 1} + z^{2q + 1}$ in~$K[z]$.
Note that, if for~$\beta$ in~$K$ we put~$h_{\beta}(z) \= \beta z$, then
$$ h_{\beta}^{-1} \circ Q_{\beta^q} \circ h_{\beta}(z)
=
\beta^{2q} z \left( 1 + z^q + z^{2q} \right). $$
Thus, to prove that~$P(z)$ has no parabolic periodic point, it is enough to show that if~$\alpha$ is transcendental over~$F$, then~$Q_\alpha$ has no parabolic periodic point.

Let~$m \ge 1$ be an integer, and let~$R(\alpha)$ be the resultant of the polynomials
$$ Q_{\alpha}^m(z) - z
\text{ and }
\left( Q_{\alpha}^m \right)'(z) - 1, $$
viewed as a polynomial in~$\alpha$ with coefficients in~$F$.
Note that~$Q_{\alpha}$ has a periodic point of period~$m$ and multiplier~$1$ if and only if~$R(\alpha) = 0$.
Since~$\alpha$ is transcendental over~$F$, to show that~$R(\alpha)$ is different from~$0$ it is enough to show that the polynomial~$R$ is nonzero.
Note that when~$\alpha = 0$ we have~$Q_0^m(z) = z^{(2q + 1)^m}$, and that the polynomials
$$ Q_0^m(z) - z = z^{(2q + 1)^m} - z
\text{ and }
\left( Q_0^m \right)'(z) - 1 = (2q + 1)^m z^{(2q + 1)^m - 1} - 1 $$
 have no common zero.
This implies that~$R(0)$ is different from~$0$, and therefore that~$R$ is nonzero.
We conclude that~$R(\alpha)$ is different from~$0$, and that~$Q_{\alpha}$ has no periodic point of period~$m$ and multiplier~$1$.
Since~$m \ge 1$ is arbitrary, we conclude that~$Q_{\alpha}$, and hence~$P$, has no parabolic periodic point.

\partn{Case 2}
The characteristic~$p$ of~$K$ is positive.
Let~$m \ge 1$ be an integer, and let~$R(\lambda)$ be the resultant of the polynomials
\begin{equation*}
P^m(z) - z
\text{ and }
(P^m)'(z) - 1,
\end{equation*}
viewed as a polynomial in~$\lambda$ with coefficients in~$F$.
Note that~$P$ has a periodic point of period~$m$ and multiplier~$1$ if and only if~$R(\lambda) = 0$.
Since~$\lambda$ is transcendental over the prime field of~$K$, to prove that~$R(\lambda)$ is different from~$0$ it is enough to show that the polynomial~$R$ is nonzero.
To do this, endow~$K$ with a non-trivial norm~$| \cdot |$, let~$\lambda_0$ in~$K$ be such that~$|\lambda_0| > 1$, and let~$P_0(z)$ be the polynomial in~$K[z]$ defined in the same way as~$P(z)$, but with~$\lambda$ replaced by~$\lambda_0$.
We show below, in several cases, that every periodic point of~$P_0$ is repelling.
This implies that~$P_0$ has no parabolic periodic point, and therefore that~$R(\lambda_0)$ is different from~$0$.
In turn this implies that~$R$ is nonzero, that~$R(\lambda)$ is different from~$0$, and that~$P$ has no periodic point of period~$m$ and multiplier~$1$.
Since~$m \ge 1$ is an arbitrary integer, this completes the proof of the lemma.

\partn{Case 2.1}
$P_0(z) = \lambda_0 z \left( 1 + z^q \right)$.
To prove that every periodic point of~$P_0(z) \= \lambda_0 z \left( 1
  + z^q \right)$ is repelling, it is enough to show that for every
periodic point~$w$ of~$P_0$ we have~$\left| P_0'(w) \right| > 1$.
Note first that if~$w$ in~$K$ satisfies~$|w| > 1$, then
$$ |P_0(w)| = |\lambda_0| \cdot |w|^{q + 1} > |w| > 1. $$
Repeating this argument we obtain that for every integer~$\ell \ge 1$ we have~$|P^{\ell}_0(w)| > |w| > 1$, so~$w$ cannot be periodic.
On the other hand, if~$w$ is in~$\OK$ and
$$ |w| = \left| 1 + w^q \right| = 1, $$
then~$|P_0(w)| = |\lambda_0| > 1$, so by the previous consideration~$w$ cannot be periodic either.
This proves that every periodic point~$w$ of~$P_0$ we have either
$$ |w| < 1
\text{ or }
\left| 1 + w^q \right| < 1. $$
If~$|w| < 1$, then
$$ \left| P_0'(w) \right|
=
|\lambda_0| \cdot \left| 1 + (q + 1) w^q \right|
=
|\lambda_0| > 1. $$
Otherwise~$\left| 1 + w^q \right| < 1$, so~$|w^q| = 1$,
$$ \left| P_0'(w) - \lambda_0 q w^q \right|
=
|\lambda_0| \cdot |1 + w^q|
<
|\lambda_0|, $$
and therefore~$|P_0'(w)| = |\lambda_0| > 1$.
This completes the proof that every periodic point of~$P_0$ is repelling.

\partn{Case 2.2}
$P_0(z) = \lambda_0 z \left( 1 + z^q + z^{2q} \right)$ and~$p \neq 3$.
As in Case~$2.1$, to prove that every periodic point of~$P_0$ is
repelling, we prove that for every periodic point~$w$ of~$P_0$ we
have~$\left| P_0'(w) \right| > 1$.
Note first that if~$w$ is in~$K$ and~$|w| > 1$, then
$$ |P_0(w)| = |\lambda_0| \cdot |w|^{2q + 1} > |w| > 1, $$
so~$w$ cannot be periodic.
On the other hand, if~$w$ is in~$\OK$ and
$$ |w|
=
\left| 1 + w^q + w^{2q} \right|
=
1, $$
then~$|P_0(w)| = |\lambda_0| > 1$, so by the previous consideration~$w$ cannot be periodic either.
This proves that each periodic point~$w$ of~$P_0$ satisfies either
$$ |w| < 1
\text{ or }
\left| 1 + w^q + w^{2q} \right|
<
1. $$
If~$|w| < 1$, then
$$ |P_0'(w)|
=
|\lambda_0| \cdot \left| 1 + (q + 1) w^q + (2q + 1) w^{2q} \right|
=
|\lambda_0|
>
1. $$
Suppose~$\left| 1 + w^q + w^{2q} \right| < 1$, and note that, together with our assumption~$p \neq 3$, this implies~$|1 + 2 w^q| = 1$.
On the other hand,
$$ \left| P_0'(w) - \lambda_0 q w^q (1 + 2 w^q) \right|
=
|\lambda_0| \cdot \left| 1 + w^q + w^{2q} \right|
<
|\lambda_0|, $$
so~$\left| P_0'(w) \right| = |\lambda_0| > 1$.
This completes the proof that every periodic point of~$P_0$ is repelling.

\partn{Case~2.3}
$P_0(z) = \lambda_0 z \left( 1 + z^q + z^{2q} \right)$ and~$p = 3$.
Note that
$$ P_0(z) = \lambda_0 z (1 - z^q)^2, $$
and that the fixed point~$z = 0$ of~$P_0$ is repelling.

To prove that every periodic point of~$P_0$ is repelling, consider the function~$\varrho \colon K \to (0, + \infty]$ defined by
$$ \varrho (z) \= \frac{1}{|z|^{\frac{2}{3}} \cdot |1 - z^q|^{\frac{1}{3}}}, $$
viewed as a singular metric~$\varrho$ on~$K$.
Below we show that for every periodic point~$w_0$ of~$P_0$ different from~$0$, the derivative
$$ \left| P_0' \right|_{\varrho}(w_0)
\=
\left| P_0'(w_0) \right| \frac{\varrho(P_0(w_0))}{\varrho(w_0)} $$
is finite and satisfies~$\left| P_0' \right|_{\varrho}(w_0) > 1$.
Denoting the orbit of~$w_0$ by~$\cO$, this implies that the
multiplier~$\prod_{w \in \cO} P'_0(w)$ of~$w_0$ satisfies
$$ \left| \prod_{w \in \cO} P_0'(w) \right|
=
\prod_{w \in \cO} \left( \left| P'_0(w) \right|
  \frac{\varrho(P_0(w))}{\varrho(w)} \right)
=
\prod_{w \in \cO} \left| P'_0 \right|_{\varrho} (w)
>
1, $$
so~$w_0$ is repelling.
Since~$z = 0$ is a repelling fixed point of~$P_0$, it follows that every periodic point of~$P_0$ is repelling.

Let~$w_0$ be a periodic point of~$P_0$ different from~$0$, and let~$\cO$ be its orbit.
Note that every element~$w$ of~$\cO$ is different from~$0$.
On the other hand, no element of~$\cO$ can be a zero of~$1 - z^q$, because every zero of~$1 - z^q$ is mapped to~$0$ by~$P_0$.
Thus, $\varrho$ is finite on~$\cO$, and therefore~$\left| P_0'
\right|_{\varrho}$ is also finite on~$\cO$; in particular, $\left|
  P_0' \right|_{\varrho}(w_0)$ is finite.
It remains to prove~$\left| P_0' \right|_{\varrho}(w_0) > 1$.
To do this, we prove first that for each~$w$ in~$\cO$ we have either~$|w| < 1$ or~$|1 - w^q| < 1$.
Suppose by contradiction that for some~$w$ in~$\cO$ we have~$|w| > 1$.
This implies that
$$ |P_0(w)| = |\lambda| \cdot |w|^{2q + 1} > |w| > 1. $$
Repeating this argument, we conclude that for every integer~$\ell \ge
1$ we have~$\left| P_0^{\ell}(w) \right| > |w| > 1$, so~$w$ cannot be periodic.
This contradiction proves that~$\cO$ is contained in~$\OK$.
Suppose by contradiction that for some~$w$ in~$\cO$ we have~$|w| = |1 - w^q| = 1$.
Then~$|P_0(w)| = |\lambda_0| > 1$, so by the previous consideration~$w$ cannot be periodic.
This contradiction proves that for every~$w$ in~$\cO$ we have either~$|w| < 1$ or~$|1 - w^q| < 1$.
To prove~$\left| P_0' \right|_{\varrho}(w_0) > 1$, suppose first~$|w_0| < 1$.
Note that~$P_0(w_0)$ is in~$\cO$, and therefore in~$\OK$.
On the other hand, we have
$$ |P_0(w_0)| = |\lambda_0| \cdot |w_0|,
\text{ and }
\left| P_0'(w_0) \right| = |\lambda_0|, $$
so
\begin{multline*}
\left| P_0' \right|_{\varrho}(w_0)
=
|\lambda_0| \cdot \left( \frac{|w_0|}{|P_0(w_0)|}
\right)^{\frac{2}{3}} \cdot \frac{1}{\left| 1 - P_0(w_0)^q \right|^{\frac{1}{3}}}
\\ =
|\lambda_0|^{\frac{1}{3}} \cdot \frac{1}{\left| 1 - P_0(w_0)^q \right|^{\frac{1}{3}}}
\ge
|\lambda_0|^{\frac{1}{3}}
> 1.
\end{multline*}
It remains to consider the case where~$\left| 1 - w_0^q \right| < 1$.
Then there is a zero~$\zeta_q$ of~$1 - z^q$, such that~$\varepsilon \= w_0 - \zeta_q$ satisfies~$|\varepsilon| < 1$.
Note that
$$ \left| 1 - w_0^q \right| = |\varepsilon|,
|P_0(w_0)| = |\lambda_0| \cdot |\varepsilon|^2,
\text{ and }
\left| P_0'(w_0) \right| = |\lambda_0| \cdot |\varepsilon|. $$
So, we have
\begin{multline*}
\left| P_0' \right|_{\varrho}(w_0)
=
|\lambda_0| \cdot |\varepsilon| \cdot
\frac{1}{|P_0(w_0)|^{\frac{2}{3}}} \cdot \left(\frac{|1 -
    w_0^q|}{\left| 1 - P_0(w_0)^q \right|} \right)^{\frac{1}{3}}
\\ =
|\lambda_0|^{\frac{1}{3}} \cdot \frac{1}{\left| 1 - P_0(w_0)^q \right|^{\frac{1}{3}}}
\ge
|\lambda_0|^{\frac{1}{3}}
>
1.
\end{multline*}
This completes the proof of the lemma.
\end{proof}

\begin{proof}[Proof of Proposition~\ref{p:odd optimality}]
Suppose~$p$ is odd.
Theorem~\ref{t:characterization of minimally ramified} implies that~$\tP$ is minimally ramified.
Thus, $\wideg(P^q(z) - z) = q + 1$, and for every integer~$n \ge 1$ we have
$$ \wideg \left( \frac{ P^{qp^n}(z) - z }{P^{qp^{n - 1}}(z) - z} \right) = qp^n. $$
Then the desired assertions are a direct consequence of Proposition~\ref{p:detailed optimality criterion} and Lemma~\ref{l:concrete independence}.

It remains to prove part~$1$ when~$p = 2$.
Note that our hypotheses imply that~$q$ is odd, and that~$P(z) =
\lambda \left( 1 + z^q + z^{2q} \right)$.
By Lemma~\ref{l:q iterate} with~$p = 2$ and~$\ell = q$, we have~$i_0(\tP^q) = q$.
Then~$\wideg \left( P^q(z) - z \right) = q + 1$, and the desired assertion is given by part~$1$ of Proposition~\ref{p:detailed optimality criterion}.
This completes the proof of the proposition.
\end{proof}

\begin{proof}[Proof of Proposition~\ref{p:even periodic optimality}]
By Lemma~\ref{l:q iterate} with~$p = 2$ and~$\ell = q$ we
have~$i_0(\tP) = 2q$.
So by part~$2$ of Proposition~\ref{p:almost minimally ramified}, $\tP$ is almost minimally ramified.
It follows that for every integer~$n \ge 1$ we have
$$ \wideg \left( \frac{ P^{2^n q}(z) - z }{P^{2^{n - 1} q}(z) - z}
\right)
=
2^n q. $$
Thus, in view of part~$2$ of Proposition~\ref{p:detailed optimality criterion}, it is enough to prove that~$P$ has no parabolic periodic point.
If the characteristic of~$K$ is zero, this is given by Lemma~\ref{l:concrete independence} with~$p = 2$ and~$q$ replaced by~$2q$.
Suppose the characteristic of~$K$ is~$2$.
Let~$m \ge 1$ be an integer, and let~$R(\lambda, \mu)$ be the resultant of the polynomials
\begin{equation*}
Q^m(z) - z
\text{ and }
(Q^m)'(z) - 1,
\end{equation*}
viewed as a polynomial in~$\lambda$ and~$\mu$ with coefficients in the prime field~$F$ of~$K$.
Note that~$Q$ has a periodic point of period~$m$ and multiplier~$1$ if and only if~$R(\lambda, \mu) = 0$.
Since~$\mu$ is algebraically independent with respect to~$\lambda$ over~$F$, to prove that~$R(\lambda, \mu)$ is different from~$0$ it is enough to show that the polynomial~$R$ is nonzero.
To do this, let~$P(z)$ be the polynomial in~$K[z]$ defined by
$$ P(z) = \lambda \left( 1 + z^q + z^{2q} \right). $$
By Lemma~\ref{l:concrete independence} this polynomial has no parabolic periodic point.
This implies that~$R(\lambda, 1)$ is different from~$0$, and therefore that~$R$ is nonzero.
Since~$m \ge 1$ is arbitrary, this completes the proof of the proposition.
\end{proof}

\begin{proof}[Proof of Theorem~\ref{t:quadratic optimality and minimality}]
Part~$2$ is given by Corollary~\ref{c:optimality implies minimality}.
To prove part~$1$, note that the hypotheses imply that for every~$n \ge 1$ we have
$$ \wideg \left( \frac{ P_{\lambda}^{qp^n}(z) - z }{P_{\lambda}^{qp^{n - 1}}(z) - z} \right) = qp^n. $$
So, in this case, the assertions of the theorem are a direct consequence of Lemma~\ref{l:concrete independence} with~$q = 1$ and part~$2$ of Proposition~\ref{p:detailed optimality criterion}.

\end{proof}

\appendix
\section{Normalized polynomials without periodic points of high minimal period}
\label{s:all minimal}
In this appendix we give examples of normalized polynomials having no
periodic point of high minimal period (Proposition~\ref{p:all minimal}).
A consequence is that a polynomial of the form~$f(z) = \lambda z
\left( 1 + z^{2q} \right)$ cannot be used to show that inequality~\eqref{e:periodic bound} is sharp when~$p = 2$ and~$K$ is of characteristic~$2$, see Remark~\ref{r:non-example appendix} below.
\begin{prop}
\label{p:all minimal}
Let~$p$ be a prime number and~$(K, |\cdot|)$ an ultrametric field of characteristic~$p$.
Moreover, let~$\lambda$ in~$K$ be such that~$|\lambda| = 1$ and such that the order~$q$ of~$\tlambda$ in~$\tK^*$ is finite, and let~$S(z)$ be a polynomial in~$\OK[z]$ such that~$S(0) = 1$ and such that~$\wideg(S(z) - 1)$ is finite.
If in addition~$\lambda^q \neq 1$, then the minimal period of every periodic point of
$$ Q(z) \= \lambda z S(z)^p $$
in~$\mK \setminus \{ 0 \}$ is equal to~$q$.
Furthermore, the set~$F$ of all such points is non-empty and finite, and for each~$a$ in~$F$ the multiplicity~$m_a$ of~$a$ as a fixed point of~$Q^q$ is finite and divisible by~$p$, and for every integer~$n \ge 1$ the multiplicity of~$a$ as a fixed point of~$Q^{qp^n}$ is equal to~$p^n m_a$.
\end{prop}
Note that the hypotheses of this proposition imply that~$\lambda$ is not a root of unity.
Thus~$z = 0$ is an irrationally indifferent fixed point of~$Q$ and therefore for every integer~$k \ge 1$ the multiplicity of~$z = 0$ as a fixed point of~$Q^k$ is equal to~$1$.
\begin{rema}
\label{r:non-example appendix}
Letting $p = 2$ and~$S(z) = 1 + z^q$ in Proposition~\ref{p:all
  minimal}, we obtain that for every integer~$n \ge 1$ the
polynomial~$Q(z) = \lambda z \left( 1 + z^{2q} \right)$ has no periodic point of minimal period equal to~$qp^n$.
This shows that the conclusions of Proposition~\ref{p:even periodic optimality} are false if we let~$\mu = 0$.
\end{rema}
\begin{rema}
\label{r:finitely many}
For~$Q$ as in Proposition~\ref{p:all minimal}, the set~$\mK$ is an indifferent component of the Fatou set of~$Q$.\footnote{This follows from the fact that~$Q$ has integer coefficients and that its reduction is of degree at least~$2$, see for example~\cite[Propositions~$3.18$ and~$5.2$]{Rivthese}.}
So Proposition~\ref{p:all minimal} shows that~$Q$ has an indifferent
component of the Fatou set that only has finitely many periodic
points.\footnote{Other examples of such Fatou components, which
  however only contain parabolic periodic points, can be obtained as
  follows. Consider a polynomial~$P(z)$ in~$K(z)$ of degree at
  least~$2$ whose coefficients are algebraic over the prime field
  of~$K$, and such that~$z = 0$ is an indifferent fixed point
  of~$P$. Then~$P$ has good reduction and therefore~$\mK$ is a Fatou
  component of~$P$, see for
  example~\cite[\S$4.5$]{Rivthese}. Furthermore, $z = 0$ is the only
  periodic point of~$P$ in~$\mK$. In fact, if we denote by~$q \ge 1$
  the order of~$P'(0)$, then for every integer~$k \ge 1$ the
  multiplicity of~$z = 0$ as a fixed point of~$P^{qk}$ coincides
  with~$\wideg \left( P^{qk}(z) - z \right)$. Together with Lemma~\ref{l:admissible periods}, this implies that~$z = 0$ is the only periodic point of~$P$ in~$\mK$.}
In contrast, in the $p$-adic case every indifferent component of the Fatou set contains infinitely many periodic points, see~\cite[\emph{Corollaire}~$5.13$]{Rivthese}.  
\end{rema}

The rest of this appendix is devoted to the proof of Proposition~\ref{p:all minimal}.
The following lemma is the main ingredient in the proof.
Recall that for an ultrametric field~$K$ and a polynomial~$P(z)$ in~$\OK[z]$, we use~$\left\langle P(z) \right\rangle$ to denote the ideal of~$\OK[z]$ generated by this polynomial.
\begin{lemm}
\label{lemmapowerofpgeneralT}
Let~$K$ be an ultrametric field characteristic~$p$, let~$A(z)$ and~$T_0(z)$ be polynomials in~$\OK[z]$, and put
\[
Q_{\dag}(z)
\=
z \left( 1 + A(z) \cdot T_0(z) \right)^p.
\]
Then for each integer $n \geq 1$ there exists a polynomial~$T_n(z)$ in~$\OK[z]$ such that
\begin{equation}
\label{defRngeneralT}
Q_{\dag}^{p^n}(z)
=
z \left( 1 + A(z)^{p^n} \cdot T_n(z) \right)^p.
\end{equation}
Moreover, $T_1(0) = T_0(0)^p$, and for~$n \ge 2$ we have in~$\OK[z]$
\begin{equation}
\label{superformulageneralT}
T_n (z)
\equiv
T_{n - 1}(z)^p  \mod  \left\langle z A(z)^{p^{n - 1}} \right\rangle.
\end{equation}
\end{lemm}
The proof of this lemma is below, after the following one.
\begin{lemm}
\label{lemmatildeQk}
Let~$K$ be an ultrametric field of characteristic~$p$, let~$U_{\#}(z)$ be a polynomial in~$\OK[z]$, and put~$Q_{\#}(z) \= z U_{\#}(z)^p$.
Then for each integer $k\geq 1$ there is a polynomial~$U(z)$ in~$\OK[z]$ such that~$Q_{\#}^k(z) = z U(z)^p$.
\end{lemm}
\begin{proof}
We proceed by induction in~$k$.
The desired property is satisfied with~$k = 1$ by definition of~$Q_{\#}$.
Given an integer~$k \ge 1$, suppose there is a polynomial~$U(z)$ in~$\OK[z]$ such that~$Q_{\#}^k(z) = z U(z)^p$.
Then
\[
Q_{\#}^{k+1}(z)
=
Q_{\#}^{k}(z) U_{\#} \left( Q_{\#}^k(z) \right)^p
=
z \left( U(z) U_{\#} \left( Q_{\#}^k(z) \right) \right)^p.
\]
This completes the proof of the induction step and of the lemma.
\end{proof}
\begin{proof}[Proof of Lemma~\ref{lemmapowerofpgeneralT}]
We prove the first assertion of the lemma by induction.
Let~$n \ge 0$ be an integer for which there is a polynomial~$T_n(z)$ in~$\OK[z]$ satisfying~\eqref{defRngeneralT}.
We prove that there is a polynomial~$T_{n + 1}$ in~$\OK[z]$ satisfying~\eqref{defRngeneralT} with~$n$ replaced by~$n + 1$.
To do this, we prove by induction that for every integer~$j \ge 1$, we have
\begin{equation}\label{QjpnproofTn}
Q_{\dag}^{jp^n}(z)
\equiv
z \prod_{k=0}^{j-1} \left[ 1 + A(z)^{p^n} \cdot T_n \left( Q_{\dag}^{kp^n}(z) \right) \right]^p
\mod \left\langle z^{p + 1} A(z)^{p^{2n+2}}\right\rangle.
\end{equation}
By the first induction hypothesis this holds for $j=1$.
Suppose it holds for an integer~$j \ge 1$.
Using the first induction hypothesis again, we have
\begin{equation}
\label{e:inductive of power iterate}
Q_{\dag}^{(j+1)p^{n}}(z)
=
Q_{\dag}^{jp^n}(z) \left[ 1 + A \left( Q_{\dag}^{jp^n}(z) \right)^{p^n} \cdot T_n \left( Q_{\dag}^{jp^n}(z) \right) \right]^p.
\end{equation}
On the other hand, by the second induction hypothesis we have
\begin{equation}
\label{e:coarsest approximation}
Q_{\dag}^{jp^{n}}(z)
\equiv
z
\mod \left\langle z A(z)^{p^{n+1}} \right\rangle.
\end{equation}
Consequently
\begin{equation*}
\begin{split}
A \left( Q_{\dag}^{jp^n}(z) \right)
\equiv
A(z)
\mod \left\langle z A(z)^{p^{n+1}} \right\rangle,
\end{split}
\end{equation*}
and therefore
\begin{equation*}
\begin{split}
A \left( Q_{\dag}^{jp^n}(z) \right)^{p^{n}}
\equiv
A(z)^{p^{n}}
\mod \left\langle z A(z)^{p^{2n+1}}\right\rangle.
\end{split}
\end{equation*}
Together with~\eqref{QjpnproofTn} and~\eqref{e:inductive of power iterate}, we obtain~\eqref{QjpnproofTn} with~$j$ replaced by~$j + 1$.
This proves the induction step of the second induction, and shows that~\eqref{QjpnproofTn} holds for every integer~$j \ge 1$.
To complete the proof of the induction step of the first induction, note that for each integer~$k \ge 1$ we have by~\eqref{e:coarsest approximation} with~$j$ replaced by~$k$
\[
T_n \left( Q_{\dag}^{kp^n}(z) \right)
\equiv 
T_n(z)
\mod \left\langle z A(z)^{p^{n+1}} \right\rangle.
\]
So, if we put~$m_n \= \min \left\{ p^{2n + 2}, p^{n + 2} + p^{n + 1} \right\}$, then by~\eqref{QjpnproofTn} we have
$$ Q_{\dag}^{j p^{n}}(z)
\equiv
z \left( 1 + A(z)^{p^{n}} \cdot T_n(z) \right)^{jp}
\mod \left\langle z^{p + 1} A(z)^{m_n} \right\rangle. $$
Taking~$j = p$ we obtain,
$$ Q_{\dag}^{p^{n + 1}}(z)
\equiv
z \left( 1 + A(z)^{p^{n + 1}} \cdot T_n(z)^p \right)^{p}
\mod \left\langle z^{p + 1} A(z)^{m_n}\right\rangle. $$
Since~$m_n \ge p^{n + 2}$, using Lemma~\ref{lemmatildeQk} with~$U_{\#}(z) = 1 + A(z)T_0(z)$ and~$k = p^{n + 1}$ we conclude that there is a polynomial~$T_{n+1}(z)$ in~$\OK[z]$ for which~\eqref{defRngeneralT} is satisfied with~$n$ replaced by~$n + 1$.
We have thus completed the proof of the first induction step and, as a consequence, showed that~\eqref{defRngeneralT} holds for every integer~$n \ge 1$.

When~$n = 0$ we have~$m_0 = p^2$ and the last displayed equation implies that~$T_1(0) = T_0(0)^p$.
On the other hand, note that for~$n \ge 2$ we have~$m_{n - 1} = p^{n + 1} + p^n$, so the last displayed equation with~$n$ replaced by~$n - 1$ combined with Lemma~\ref{lemmatildeQk} with~$U_{\#}(z) = 1 + A(z) T_0(z)$ and~$k = p^n$, shows~\eqref{superformulageneralT} for~$n \ge 2$.
This completes the proof of the lemma.
\end{proof}
\begin{proof}[Proof of Proposition~\ref{p:all minimal}]
Extending~$K$ if necessary, assume~$K$ is algebraically closed and complete with respect to~$| \cdot |$.
Let~$\eta$ in~$K$ be such that~$\eta^p = \lambda$, so that
$$ Q(z) = z \left( \eta S(z) \right)^p
\text{ and }
\frac{Q(z) - z}{z} = \left( \eta S(z) - 1 \right)^p. $$
In the case~$q \ge 2$, it follows that the polynomial~$Q$ has no fixed point in~$\mK \setminus \{ 0 \}$.
Thus, in view of Lemma~\ref{l:admissible periods}, to prove that the minimal period of every periodic point of~$Q$ in~$\mK \setminus \{ 0 \}$ is equal to~$q$, we just need to show that for every integer~$n \ge 1$, every fixed point of~$Q^{qp^n}$ in~$\mK$ is also a fixed point of~$Q^q$.
To do this, note that by Lemma~\ref{lemmatildeQk} with~$U_{\#}(z) = \eta S(z)$ and~$k = q$ there is a polynomial~$U(z)$ in~$\OK[z]$ such that~$Q^q(z) = z U(z)^p$.
Our hypothesis~$S(0) = 1$ implies that~$U(0) = 1$.
On the other hand, our hypothesis that~$\wideg(S(z) - 1)$ is finite implies that~$\wideg(Q)$, and hence~$\wideg(Q^q)$, are both finite and greater than or equal to~$2$.
In turn, this implies that~$\wideg(U(z) - 1)$ is finite and nonzero.
Thus the set~$F$ of fixed points of~$Q^q$ in~$\mK \setminus \{ 0 \}$ is non-empty and finite, and for each~$a$ in~$F$ the multiplicity~$m_a$ of~$a$ as a fixed point of~$Q^q$ is finite and divisible by~$p$.
Put
$$ A(z) \= \prod_{a \in F} (z - a)^{\frac{m_a}{p}}, $$
and note that by applying Lemma~\ref{l:elementary division} repeatedly, it follows that there is a polynomial~$T_0(z)$ in~$\OK(z)$ such that
$$ |T_0(0)| = 1
\text{ and }
Q^q(z) = z \left( 1 + A(z) \cdot T_0(z) \right)^p. $$
So we can apply Lemma~\ref{lemmapowerofpgeneralT} with~$Q_{\dag} = Q^q$.
We obtain that for each integer~$n \ge 1$ there is a polynomial~$T_n(z)$ in~$\OK[z]$ such that
\begin{equation}
\label{e:power iterate}
Q^{qp^n}(z)
=
z \left( 1 + A(z)^{p^n} \cdot T_n(z) \right)^p.
\end{equation}
Moreover, by an induction argument we conclude that for every~$n \ge 1$ we have~$T_n(0) = T_0(0)^{p^n}$.
In particular, for every integer~$n \ge 1$ we have~$|T_n(0)| = 1$.
This implies that every fixed point~$Q^{qp^n}$ in~$\mK \setminus \{ 0 \}$ is a zero of~$A$ and therefore a fixed point of~$Q^q$.
Furthermore, for every zero~$a$ of~$A$, the multiplicity of~$z = a$ as a zero of~$Q^{qp^n}(z) - z$ is equal to~$p^nm_a$.
This completes the proof of the proposition.
\end{proof}

\bibliographystyle{alpha}

\end{document}